\documentclass[12pt,dvips]{amsart}
\usepackage{amsfonts, amssymb, latexsym, epsfig, amsthm}

  \theoremstyle{plain}
    \newtheorem{theorem}{Theorem}[section]

  \newtheorem{proposition}[theorem]{Proposition}
  
  \newtheorem{lemma}[theorem]{Lemma}
  
  \newtheorem{corollary}[theorem]{Corollary}
  \newtheorem{conjecture}[theorem]{Conjecture}
  \newtheorem{problem}[theorem]{Problem}
  \newtheorem{claim}[theorem]{Claim}
\theoremstyle{definition}
  \newtheorem{definition}[theorem]{Definition}
  \newtheorem{example}[theorem]{Example}
  
  \newtheorem{question}[theorem]{Question}

  \newtheorem{remark}[theorem]{Remark}
\newcommand{\excise}[1]{}

\newcommand{\rdots}{\mathinner{%
  \mkern1mu\raise1pt\hbox{.}%
  \mkern2mu\raise4pt\hbox{.}%
  \mkern2mu\raise7pt\vbox{\kern7pt\hbox{.}}\mkern1mu}}

\numberwithin{equation}{section}

\def\CC{{\mathbb C}}

\def\Pet{{\rm Pet}_n}
\def\zz{ z_{ab} }
\def\bfz{ {\bf z} }
\def\module{ {\bf M} }

\def\matrx{ K }
\def\nilp{H }

\def\uu{\eta}

\newcommand{\cellsize}{15}
\newlength{\cellsz} 
\newsavebox{\cell}
\sbox{\cell}{\begin{picture}(\cellsize,\cellsize)
\put(0,0){\line(1,0){\cellsize}}
\put(0,0){\line(0,1){\cellsize}}
\put(\cellsize,0){\line(0,1){\cellsize}}
\put(0,\cellsize){\line(1,0){\cellsize}}
\end{picture}}
\newcommand\cellify[1]{\def\thearg{#1}\def\nothing{}%
\ifx\thearg\nothing
\vrule width0pt height\cellsz depth0pt\else
\hbox to 0pt{\usebox{\cell} \hss}\fi%
\vbox to \cellsz{
\vss
\hbox to \cellsz{\hss$#1$\hss}
\vss}}
\newcommand\tableau[1]{\vtop{\let\\\cr
\baselineskip -16000pt \lineskiplimit 16000pt \lineskip 0pt
\ialign{&\cellify{##}\cr#1\crcr}}}
\newcommand{\kellsize}{18}
\newlength{\kellsz} 
\newsavebox{\kell}
\sbox{\kell}{\begin{picture}(\kellsize,\kellsize)
\put(0,0){\line(1,0){\kellsize}}
\put(0,0){\line(0,1){\kellsize}}
\put(\kellsize,0){\line(0,1){\kellsize}}
\put(0,\kellsize){\line(1,0){\kellsize}}
\end{picture}}
\newcommand\kellify[1]{\def\thearg{#1}\def\nothing{}%
\ifx\thearg\nothing
\vrule width0pt height\kellsz depth0pt\else
\hbox to 0pt{\usebox{\kell} \hss}\fi%
\vbox to \kellsz{
\vss
\hbox to \kellsz{\hss$#1$\hss}
\vss}}
\newcommand\ktableau[1]{\vtop{\let\\\cr
\baselineskip -16000pt \lineskiplimit 16000pt \lineskip 0pt
\ialign{&\kellify{##}\cr#1\crcr}}}
\font\co=lcircle10

\def\jr{\smash{\raise2pt\hbox{\co \rlap{\rlap{\char'005} \char'007}}
               \raise6pt\hbox{\rlap{\vrule height5pt}}
               \raise2pt\hbox{\rlap{\hskip4pt \vrule height0.4pt depth0pt
                width5.7pt}}
               \raise2pt\hbox{\rlap{\hskip-9.5pt \vrule height.4pt depth0pt
                width6.2pt}}
               \lower6pt\hbox{\rlap{\vrule height4.5pt}}}}
\def\rj{\smash{\raise2pt\hbox{\co \rlap{\rlap{\char'004} \char'006}}
               \raise6pt\hbox{\rlap{\vrule height5pt}}
               \raise2pt\hbox{\rlap{\hskip4pt \vrule height0.4pt depth0pt
                width5.7pt}}
               \raise2pt\hbox{\rlap{\hskip-9.5pt \vrule height.4pt depth0pt
                width6.2pt}}
               \lower6pt\hbox{\rlap{\vrule height4.5pt}}}}
\def\je{\smash{\raise2pt\hbox{\co \rlap{\rlap{\char'005}
                \phantom{\char'007}}}\raise6pt\hbox{\rlap{\vrule height5pt}}
               \raise2pt\hbox{\rlap{\hskip-9.5pt \vrule height.4pt depth0pt
                width6.2pt}}}}
\def\ej{\smash{\raise2pt\hbox{\co \rlap{\rlap{\char'004}\phantom{\char'006}}}
               \raise2pt\hbox{\rlap{\hskip-9.5pt \vrule height.4pt depth0pt
                width6.2pt}}
               \lower6pt\hbox{\rlap{\vrule height4.5pt}}}}
\def\er{\smash{\raise2pt\hbox{\co \rlap{\rlap{\phantom{\char'005}} \char'007}}
               \raise2pt\hbox{\rlap{\hskip4pt \vrule height0.4pt depth0pt
                width5.7pt}}
               \lower6pt\hbox{\rlap{\vrule height4.5pt}}}}
\def\re{\smash{\raise2pt\hbox{\co \rlap{\rlap{\phantom{\char'004}} \char'006}}
               \raise6pt\hbox{\rlap{\vrule height5pt}}
               \raise2pt\hbox{\rlap{\hskip4pt \vrule height0.4pt depth0pt
                width5.7pt}}}}
\def\+{\smash{\lower6pt\hbox{\rlap{\vrule height17pt}}
                \raise2pt
                \hbox{\rlap{\hskip-9pt \vrule height.4pt depth0pt
                width18.7pt}}}}
\def\hor{\smash{\raise2pt\hbox{\rlap{\hskip-9.5pt \vrule height.4pt depth0pt
                width19.2pt}}}}
\def\ver{\smash{\lower6pt\hbox{\rlap{\vrule height17pt}}}}
\def\ho{\smash{\hbox{\rlap{\vrule height5pt}}
                \raise2pt
                \hbox{\rlap{\hskip-9pt \vrule height.4pt depth0pt
                width18.7pt}}}}

\def\textcross{\ \smash{\lower4pt\hbox{\rlap{\hskip4.15pt\vrule height14pt}}
                \raise2.8pt\hbox{\rlap{\hskip-3pt \vrule height.4pt depth0pt
                width14.7pt}}}\hskip12.7pt}
\def\textelbow{\ \hskip.1pt\smash{\raise2.75pt%
                \hbox{\co \hskip 4.15pt\rlap{\rlap{\char'004} \char'006}
                \lower6.8pt\rlap{\vrule height3.5pt}
                \raise3.6pt\rlap{\vrule height3.5pt}}
                \raise2.8pt\hbox{%
                  \rlap{\hskip-7.15pt \vrule height.4pt depth0pt width3.5pt}%
                  \rlap{\hskip4.05pt \vrule height.4pt depth0pt width3.5pt}}}
                \hskip8.7pt}

\begin{document}
\pagestyle{plain}
\title{Patch ideals and Peterson varieties}
\author{Erik Insko}
\address{Department of Chemistry and Mathematics\\
Florida Gulf Coast University\\
Fort Myers, FL 33965-6565}
\email{einsko@fgcu.edu}
\author{Alexander Yong}
\address{Dept. of Mathematics\\
University of Illinois at Urbana-Champaign\\
Urbana, IL 61801}
\email{ayong@math.uiuc.edu}

\date{August 7, 2012}

\begin{abstract}
\emph{Patch ideals} encode neighbourhoods of a variety in $GL_n/B$.
For Peterson varieties we determine generators for these ideals and show
they are complete intersections,
and thus Cohen-Macaulay and Gorenstein. Consequently, we
\begin{itemize}
\item combinatorially describe the singular locus of the Peterson variety;
\item give an explicit equivariant $K$-theory localization formula; and
\item extend some results of [B.~Kostant '96] and of
D.~Peterson to intersections of Peterson varieties with Schubert varieties.
\end{itemize}
We conjecture that the tangent cones are Cohen-Macaulay,
and that their $h$-polynomials are nonnegative and upper-semicontinuous.

Similarly, we use patch ideals to briefly analyze
other examples of torus invariant subvarieties of $GL_n/B$,
including Richardson varieties and Springer fibers.
\end{abstract}

\maketitle

\section{Introduction}

\subsection{Overview}
The {\bf Peterson variety} is defined as
\[{\rm Pet}_n=\left\{F_{\bullet}\in {\rm Flags}({\mathbb C}^n): N\cdot F_i\subset F_{i+1}\right\},\]
where $N$ is a fixed but arbitrary regular nilpotent matrix (up to isomorphism, $\Pet$ is independent of the choice of
conjugate of $N$),
and ${\rm Flags}({\mathbb C}^n)$ is the variety of complete flags of subspaces
$F_{\bullet}:\langle 0\rangle \subset F_1 \subset F_2\subset\cdots\subset F_{n-1}\subset {\mathbb C}^n$.

 Since its introduction by D.~Peterson in the 1990's, in connection to quantum cohomology of flag varieties, this space has
 attracted significant attention;
 see, for instance, \cite{kostant, reitsch, Tym07, harada.tymoczko} and the references therein.

While Peterson varieties are known to be singular for $n\geq 3$ \cite[Theorem~6]{kostant},
problems about their local structure exist:

\begin{problem}
\label{prob:A}
Give an explicit combinatorial description of the singular locus.
\end{problem}

\begin{problem}
\label{prob:B}
Determine properties (e.g., Cohen-Macaulay, Gorenstein, normal) that hold for all singular points.
\end{problem}

\begin{problem}
\label{prob:C}
Give an explicit combinatorial rule for discrete
local singularity measures
(e.g., equivariant $K$-theory localizations, Hilbert-Samuel multiplicity).
\end{problem}

Such problems raise two meta-questions about varieties $X$ defined inside $GL_n/B$.

First, what methods are available to study
local features of a given $X \subseteq GL_n/B$?
This question arises since a number of significant varieties
are defined to sit in the flag variety. Perhaps most prominent are the Schubert varieties. However, there are also the
Richardson varieties, torus orbit closures ${\overline{Tx}}$, and the family of Hessenberg varieties (which include the Springer and Peterson varieties), among others.
One desires uniform and flexible approaches for theoretical and computational analysis.

Indeed, important and influential techniques have been developed that are used to study such $X$. For example,
\emph{Standard Monomial Theory}, a survey of which is found in \cite{LakMag},
\emph{GKM-theory} \cite{GKM}, \emph{Frobenius splitting}, a recent book for which is
\cite{Brion.Kumar},
and use of the \emph{Peterson translate} \cite{Carrell.Kuttler}.
These techniques are especially effective, but not necessarily definitive, for Schubert varieties.

Therefore, we revisit this meta-question by examining the general use of a simple
choice of local coordinates and equations encoded in what we call the \emph{patch ideal}.
We emphasize that in the literature, one finds many examples of similar ideas, used for
instances of $X$ (and usually $X$ is a Schubert variety). We provide
a brief review in
Section~2, where we also
offer a synthesis of a selection of these ideas. This facilitates our common discussion of Peterson varieties,
specific earlier work on Schubert varieties \cite{WY:govern, WY:KL, Li.Yong, Li.Yong2}, and some other
$X$.
In our main example of Peterson varieties, we show how to apply these ideas to provide some answers to the above problems.

Second, what are some candidate properties for
comparing different $X \subseteq GL_n/B$?
All of the aforementioned examples of $X$
share the common feature of an action of a (sub)torus of the invertible diagonal matrices $T\subset GL_n$.
Deeper commonalities among some of them are known; a partial list of well-studied
properties is given in Problem~\ref{prob:B}.
We conjecture that a number of interesting $X$ (including Peterson and Schubert varieties) share
common properties concerning their tangent cones.

Turning to our main results, we give coordinates and equations that scheme-theoretically
cut out a natural open neighbourhood (the \emph{patch}) of any point of $X={\rm Pet}_n$. We prove
these equations define a radical ideal (the patch ideal). To do this, we show the ideal is a complete intersection and hence Cohen-Macaulay. Thus, the
zero scheme is reduced if it is generically reduced. To prove the latter, we
establish the existence of a smooth point.

Consequently, ${\rm Pet}_n$ is a local complete intersection, and therefore Cohen-Macaulay and Gorenstein. J.~Carrell informed us that at least some of these properties
(see Corollary~\ref{thm:first}) were known to D.~Peterson (unpublished);
we also offer a generalization, see Theorem~\ref{thm:lciextended}.
In addition, B.~Kostant \cite[Theorem~14]{kostant} partially described the
singular locus of ${\rm Pet}_n$. We give
a combinatorial description of the singular locus. We also obtain
an explicit equivariant $K$-theory localization formula, by approaching the question in terms
of combinatorial commutative algebra.

\subsection{Peterson varieties and the main results}
Identify
${\rm Flags}({\mathbb C}^n)\cong GL_n/B$,
where we let
$B\subset GL_n$ denote the subgroup of invertible upper triangular matrices. A coset $gB$ corresponds to
the flag $F_{\bullet}$ where $F_i$ is the subspace spanned by the vectors represented by the leftmost $i$ columns of any coset
representative of $gB$.

Let $w=w_P$ be the maximal Coxeter length element
for the parabolic (Young) subgroup
\begin{equation}
\label{eqn:weyl}
W_P\cong S_{i_1}\times S_{i_2}\times\cdots\times S_{i_k}, \mbox{\ \ \ \  $i_1+i_2+\cdots+i_k=n$}
\end{equation}
of the symmetric group $S_n$, in the latter's role as the Weyl group $W$ of $GL_n$.
Let
\begin{equation}
\label{eqn:BPdef}
U_P=B^{(1)}\oplus B^{(2)}\oplus\cdots\oplus B^{(k)}\subset U
\end{equation}
be the $(n-k)$-dimensional subgroup of unipotent upper triangular matrices $U$, consisting of block matrices
where the block $B^{(j)}$ is an $i_j\times i_j$ upper triangular unipotent matrix with equal
entries on each superdiagonal (see Example~\ref{example:BP}).
There is a decomposition
\begin{equation}
\label{eqn:thedecomposition}
{\rm Pet}_n=\coprod_{w_P}{\mathcal J}_{w_P, \Pet},
\end{equation}
of ${\rm Pet}_n$ into affine cells indexed by the parabolic subgroups $P$ contained in $B$,
where
\begin{equation}
\label{eqn:Jaffine}
{\mathcal J}_{w_P, \Pet}=U_P\cdot w_P B\cong {\mathbb C}^{n-k};
\end{equation}
see Lemma~\ref{prop:Petjust}(IV).

Our first result describes the singular locus of $\Pet$. This
answers Problem~\ref{prob:A}:

\begin{theorem}
\label{thm:singlocus}
The singular locus of $\Pet$ is given by
\[{\rm Sing}(\Pet)=\coprod_{w_P}{\mathcal J}_{w_P,\Pet},\]
where the union is over all $w_P$ satisfying any of the equivalent
conditions {\rm (I)}--{\rm (III)} below:
\begin{itemize}
\item[(I)] $w_P B$ is singular in ${\rm Pet}_n$;
\item[(II)] $w_P$ is not one of the permutations (written in one-line notation):
\begin{itemize}
\item[$\bullet$] $n \ n-1 \ n-2 \cdots 3 \ 2 \ 1 \ \mbox{(i.e., the maximal Coxeter length element $w_0\in S_n$);}$
  \item[$\bullet$] $1 \ n \ n-1 \ \cdots \ 3 \ 2$; or
  \item[$\bullet$] $n-1 \ n-2 \ \cdots 1\ n$.
\end{itemize}
\item[(III)] $w_P$ contains at least one of the patterns $123$ or $2143$.
\end{itemize}

\end{theorem}

If it were true that every $U_P$ acts on $\Pet$, the condition (I) would be obvious (see Lemma~\ref{lemma:goodcases}(I)).
However, this is not actually the case.
Although (II)$\Leftrightarrow$(III) is easy,
the pattern criterion (III) is reminiscent of the results about singularities of
Schubert varieties; see \cite[Chapter 8]{Billey.Lakshmibai} and the references therein.

The one-dimensional torus
$S\!=\!\{s(t)={\rm diag}(t,t^2,\ldots,t^n):t\in {\mathbb C}^{\star}\}\subseteq~T$
acts on $\Pet$ (see Lemma~\ref{prop:Petjust}(I)). Thus, there is an $S$-equivariant $K$-theory class
$[{\mathcal O}_{\Pet}]\in K_S(GL_n/B)$,
where ${\mathcal O}_{\Pet}$ is the ($S$-equivariant) structure sheaf of $\Pet$.
It makes sense to localize this class at the $S$-fixed points $w_P B$. This localization satisfies
\[[{\mathcal O}_{\Pet}]|_{w_P B}\in K_S(w_PB)\cong {\mathbb Z}[\chi^{\pm 1}]\]
where $\chi:S\to {\mathbb C}^{\star}$ is the
character $s(t)\mapsto t$ of $S$.

\begin{theorem}
\label{thm:localize}
The following localization formula holds:
\[[{\mathcal O}_{\rm Pet_n}]_S|_{w_P B}=\prod_{j=1}^{n-2}\prod_{k\neq w_P (i), 1\leq i\leq j+1}\left ( 1-\chi^{k+1-w_P(j)} \right ) \in K_S(w_P B)\cong {\mathbb Z}[\chi^{\pm 1}].\]
\end{theorem}

Theorem~\ref{thm:localize} addresses Problem~\ref{prob:C} for
only one of the available numerical local invariants. For example, we do not know
a formula for the Hilbert-Samuel multiplicities of ${\rm Pet}_n$. However, we will be able to at least
compute some of them (using (P.V) of Section~2). This has been effective enough to support some new conjectures about Peterson varieties.

We will deduce Theorems~\ref{thm:singlocus} and~\ref{thm:localize}
from our next result, providing local
coordinates and equations for $\Pet$.
Let $Z^{(w_P)}$ denote the generic matrix in $ w_P U_{-}$ where $U_{-}$ is the subgroup of unipotent matrices
in the group of invertible lower triangular matrices $B_{-}$. This matrix has
$z_{ij}$ placed in the usual matrix coordinate position and
\begin{equation}
\label{eqn:pivot}
z_{w_P(j),j}=1 \mbox{\ and \ } z_{ij}=0 \mbox{\ if $j>w_P^{-1}(i)$.}
\end{equation}
Let $\bfz=\{\zz\}$ be the collection of variables not specialized by (\ref{eqn:pivot}).
Also, let
$Z_j=[{z}_{1j} \ {z}_{2j} \ \cdots \ {z}_{nj}]^t$
denote the $j^{th}$ column of $Z^{(w_P)}$.
See Example~\ref{exa:2143general} below.

Fix $N$ to be the standard regular nilpotent $n\times n$ matrix consisting of $1$'s on the main superdiagonal and $0$'s elsewhere.
Consider the following system of linear equations in the unknowns $\{\alpha_{j,\ell}\}$:
\begin{equation}
\label{eqn:bigone}
N\cdot Z_j=\alpha_{j,1}Z_1+\alpha_{j,2}Z_2+\cdots +\alpha_{j,j+1}Z_{j+1}, \ \ \ \mbox{ \ for $1\leq j\leq n-2$.}
\end{equation}

For each $j$, label the equation in the $k$-th row of (\ref{eqn:bigone}) by
\begin{equation} \label{gformula} f_{k,j}: \ \ \
\alpha_{j,1} {z}_{k,1} + \cdots + \alpha_{j,j+1} {z}_{k, j+1} = {z}_{k+1,j},
\end{equation}
where $1\leq j\leq n-2$ and $1\leq k\leq n$. We use the
convention ${z}_{n+1,j}=0$.

Solving for the unknowns
$\alpha_{j,1},\alpha_{j,2},\ldots,\alpha_{j,j+1}\in {\mathbb C}[{\bfz}]$,
in that order, it is straightforward to check that
the subsystem of (\ref{eqn:bigone}), consisting of equations $f_{k,j}$ where $k=w_P(\ell)$ and
$1\leq \ell\leq j+1$, has a unique solution, recursively expressed as:
\begin{equation}
\label{eqn:alphaformula}
	\alpha_{j,\ell} = {z}_{w_P(\ell)+1,j} - \left(\sum_{t=1}^{\ell-1} \alpha_{j,t} \ {z}_{w_P(\ell),t}\right)\in {\mathbb C}[{\bfz}].
\end{equation}
Consider the ${n-1\choose 2}$ equations $f_{k,j}$ where
$k\neq w_P(\ell)$ for $1\leq \ell\leq j+1$. Let
\[g_{k,j}=\alpha_{j,1} {z}_{k,1} + \cdots + \alpha_{j,j+1} {z}_{k, j+1} -{z}_{k+1,j}\in {\mathbb C}[\bfz]\]
be the polynomial obtained from
$f_{k,j}$ after substituting for $\alpha_{j,\ell}$ with
the unique solution given by (\ref{eqn:alphaformula}). Define the ideal
\begin{equation}
\label{eqn:theideal}
I_{w_P,\Pet}=\langle g_{k,j}:k\neq w_P(\ell), 1\leq \ell\leq j+1\rangle\subseteq {\mathbb C}[\bfz]
\end{equation}
and set
$Y_{w_P,\Pet}={\rm Spec}\left({\mathbb C}[\bfz]/I_{w_P,\Pet}\right)$.

The following result, in the terminology of Section~2, shows
that $I_{w_P,\Pet}$ is a patch ideal:
\begin{theorem}
\label{thm:main1}
$Y_{w_P,\Pet}$ is reduced and it is
isomorphic to the affine open neighbourhood
${\mathcal N}_{w_P,\Pet}=w_PB_{-}B/B\cap \Pet$ of $w_P B\in\Pet$.
\end{theorem}


\begin{example}
\label{exa:2143general}
Consider the point $w_PB\in {\rm Pet}_4$ where
\[w_P=\left(\begin{matrix}
0 &1 & 0 & 0\\
1 & 0 & 0 & 0\\
0 & 0 & 0 &1 \\
0 & 0 & 1 & 0
\end{matrix}\right) \mbox{\ \ and hence \ \ }
Z^{(w_P)}=
\begin{pmatrix}   {z}_{11}  &          1  &         0    &       0 \\
   1 &          0 &          0  &         0 \\
{z}_{31} & {z}_{32} &    {z}_{33} &         1 \\       {z}_{41} &        {z}_{42}    &       1   &        0
\end{pmatrix}.\]
Here
$(\alpha_{11},\alpha_{12})=(z_{31},1-z_{31}z_{11}) \text{\ and \ }
(\alpha_{21},\alpha_{22},\alpha_{23})=(z_{32},-z_{32}z_{11},-z_{41}z_{32}+z_{42}z_{32}z_{11}))$.
Now, $g_{31},g_{41}$ and $g_{32}$ generate the ideal, i.e.,
\begin{multline}\nonumber
I_{w_P,{\rm Pet}_4}=\langle z_{41}-z_{32}-z_{31}^2+z_{31}z_{32}z_{11}, \
z_{41} z_{31}+z_{42}-z_{42}z_{31}z_{11}, \ \\ \nonumber
z_{42}-z_{31}z_{32}+z_{32}^2 z_{11}+z_{41}z_{32}z_{33}-z_{42}z_{32}z_{33}z_{11}
\rangle. \qed \end{multline}
\end{example}

The following provides a partial answer to Problem~\ref{prob:B}:

\begin{corollary}
\label{thm:first}
${\rm Pet}_n$ is a local complete intersection (and hence is Cohen-Macaulay and Gorenstein). It is normal if and only if $n \leq 3$.
\end{corollary}

\begin{proof}
The neighbourhoods ${\mathcal N}_{w_P,\Pet}$ cover $\Pet$ (see (P.I) in Section~2).
Since $\dim \Pet=n-1$,
Theorem~\ref{thm:main1} asserts there is an open neighbourhood of any $gB\in \Pet$ isomorphic to an affine variety of codimension ${n-1 \choose 2}$ in ${\mathbb C}^{n \choose 2}$. This codimension equals the number
of generators of $I_{w_P,\Pet}$. Hence $I_{w_P,\Pet}$ is a complete intersection, proving
the first statement. Since
every point is locally Cohen-Macaulay and Gorenstein,
the same is true of ${\rm Pet}_n$;
see, e.g., \cite[Prop~18.13]{Eisenbud} and \cite[Cor~21.19]{Eisenbud}.

The normality statement was proved in \cite[Theorem~14]{kostant}. Now,
${\rm Pet}_2\cong {\mathbb P}^1$ is smooth, whereas one can directly check
that ${\rm Pet}_3$ is normal at its unique singular point $(id)B$.
We recover Kostant's theorem since Theorem~\ref{thm:singlocus} asserts the singular locus is codimension one when $n\geq 4$: for
$w_P=k \ k-1 \ k-2 \ \cdots \ 1 \ n \ n-1 \ \cdots \ k+1 \ \ \ \mbox{\ and \ \ \ $k\geq 2$,}$
all points in ${\mathcal J}_{w_P, \Pet}\cong {\mathbb C}^{n-2}$
are singular.
\end{proof}

\subsection{Organization}
In Section~2, we provide background about patch ideals that we will need.
In Section~3, we make additional comparisons of our work on Peterson varieties with the literature.
In Section~4, we prove Theorem~\ref{thm:main1}, after recalling necessary facts about
Peterson varieties.
In Section~5, we prove Theorem~\ref{thm:singlocus}.
In Section 6, we prove Theorem~\ref{thm:localize} and discuss conjectures about the
tangent cone.
In Section~7, we examine some other $X\subseteq GL_n/B$. We hint that Conjecture~\ref{conj:Petmain}
from Section~6 may be more broad.
We will also state and prove Theorem~\ref{thm:lciextended} that generalizes
Theorem~\ref{thm:main1} and Corollary~\ref{thm:first} to
Peterson-Schubert varieties.

\section{Background, analysis using patch ideals (P.I)--(P.V)}

We present five rudimentary themes we use to study $X\subseteq GL_n/B$. These apply standard facts
about the flag variety \cite{Billey.Lakshmibai, Brion:notes}, equivariant $K$-theory \cite{Chriss.Ginzburg},
and combinatorial commutative algebra
\cite{Eisenbud, Miller.Sturmfels}, summarized for our setting.

\medskip
\noindent
(P.I) \emph{A choice of open neighbourhood of $X$}: Recall the two {\bf Bruhat decompositions}
\begin{equation}
\label{eqn:Bruhat}
GL_n/B=\coprod_{w\in S_n} BwB/B = \coprod_{w\in S_n} B_{-}wB/B.
\end{equation}
The opposite big cell $B_{-}B/B$ provides an affine open neighbourhood of $(id)B$ in $GL_n/B$. Therefore
\[{\mathcal N}_{g,X}=gB_{-}B/B\cap X\]
provides an affine open neighbourhood of $gB$ in $X$. This may be called an $X$ {\bf patch} (we thank A.~Knutson for the terminology).

We desire explicit coordinates and equations
for ${\mathcal N}_{g,X}$. By the first Bruhat decomposition
(\ref{eqn:Bruhat}), $gB=(bw)B$ for some $b\in U$ and $w \in S_n$.
In fact, we may further assume $w^{-1}bw \in U_{-}$ \cite[Theorem 28.4]{H}. So we may
assume $g=bw$ for such $b$ in what follows. Now, consider the projection map
\[\pi:GL_n\to GL_n/B.\]
Let $U_{-}$ be the subgroup of unipotent lower triangular matrices.
Since $\pi$ is a trivial fibration over $B_{-}B/B$ with fiber $B$,
it admits a local section
\[\sigma:B_{-}B/B\to GL_n\]
such that  $ \sigma(B_{-}B/B) = U_{-} \subset GL_n$  \cite[Section~1.2]{Brion:notes}.
Similarly, $\pi$ admits a local section \[ \sigma_g := g \sigma g^{-1}:gB_{-}B/B \to GL_n ,\] which provides a scheme-theoretic isomorphism  $\sigma_g (gB_{-}B/B) = gU_{-} \subset GL_n$, and hence each section $\sigma_g$ identifies explicit coordinates for the patch $\mathcal{N}_{g,GL_n/B} = gB_{-}B/B$.   
Restricting $\pi$ and $\sigma_g$ to $X$ one obtains:
\begin{equation}
\label{eqn:patch}
{\mathcal N}_{g,X}= gB_{-}B/B \cap X \cong \sigma_g(gB_{-}B/B\cap X)\cong \pi^{-1}(X)\cap \sigma_g (gB_{-}B/B),
\end{equation}
scheme-theoretically (the final intersection being transverse).

    A generic matrix $Z^{(w)}$ in $wU_{-}$ has $1$'s in row $w(j)$ and column $j$, and in each row $0$'s to the right of the $1$.
Let us index the unspecialized entry of $Z^{(w)}$ in row $i$ and column $j$ by $z_{ij}$. Hence a generic matrix $Z^{(bw)}$ in $bwU_{-}$
is $bZ^{(w)}$, i.e., a matrix with affine linear forms in the indeterminates ${\bf z}=\{z_{ij}\}$. This explicates the fact $bwU_{-}=wU_{-}$.
We thus identify
\begin{equation}
\label{eqn:polyident}
{\mathbb C}[bwU_{-}]\cong {\mathbb C}[{\bf z}]
\end{equation}
In view of the identifications (\ref{eqn:patch}) and (\ref{eqn:polyident}), we may
interpret ${\mathcal N}_{g,X}$ as subscheme of $\sigma_g (gB_{-}B/B)\cong bwU_{-}=wU_{-}\cong {\mathbb C}^{n\choose 2}$.

It follows that ${\mathcal N}_{bw,X}={\mathcal N}_{w,X}$.
Moreover, since the patches $\{wB_{-}B/B: w\in S_n\}$ cover $GL_n/B$, $\{{\mathcal N}_{w,X}: wB\in X\}$
covers $X\subseteq GL_n/B$. Thus, it is mostly superfluous to define ${\mathcal N}_{g,X}$ for $g\neq w$. However, to conduct some local
analysis near $gB$, we set up
coordinates/equations where the origin corresponds to $gB=(bw)B$:
\begin{definition}
The {\bf patch ideal} is
$I_{g,X}=I({\mathcal N}_{g,X})\subseteq {\mathbb C}[\bfz]$.
\end{definition}
Specifically for $I_{w,X}$,
the action of $T\subseteq GL_n$
on $wB_{-}B/B$ induces an action on $wU_{-}$, namely scale each row of a matrix independently, and then rescale each column so that the entry
in row $w(j)$ and column $j$ remains a $1$. This induces a grading on ${\mathbb C}\left[wU_{-}\right]\cong {\mathbb C}[\bfz]$, which in our coordinates assigns
\begin{equation}
\label{eqn:thegrading}
{\rm deg}(z_{ij})=t_{i}-t_{w(j)}.
\end{equation}
If $S\subseteq T$ is an algebraic subtorus that acts on $X$, the grading
on $I_{w,X}$ coming from $S$ is a coarsening of (\ref{eqn:thegrading}).

\medskip
\noindent
(P.II) \emph{Using generators for the patch ideal:}
We will assume $X$ is reduced.
If $X$ is described as a collection of flags,
it is ``typically easy'' (this is made precise in our examples)
to find equations that define an ideal that set-theoretically cuts out
${\mathcal N}_{g,X}$. One does this by
determining the equations defining $\pi^{-1}(X)\subseteq GL_n$
and demanding that the matrix entries (thought of as coordinates
in ${\mathbb C}[{\bf z}]$)
of $Z^{(bw)}=bZ^{(w)}\in GL_n$ from (P.I) satisfy these equations.
It remains to show these equations generate a radical ideal.
Two typical ways to guarantee this, and that we refer to, are:

(i): If the generators define a complete intersection, the associated
scheme is Cohen-Macaulay. Thus if the scheme is generically reduced, it is reduced
(see \cite[Exercise~18.9]{Eisenbud}). Generic reducedness is
guaranteed if $X$ is irreducible and has a smooth point.

(ii): If the generators form a Gr\"{o}bner basis with squarefree lead terms,
the initial ideal is radical, and so must have been the
original ideal (see \cite[Section~15.8]{Eisenbud}).

In our computational analysis from Sections~6 and 7 we use facts about
free resolutions to verify properties of $X$; see e.g., \cite{Bruns.Herzog, Eisenbud}.
A point $gB$ is  Cohen-Macaulay if the minimal free resolution of the local ring
$\left({\mathbb C}[\bfz]/I_{g,X}\right)_{(z_{ij})}$
is {\bf short}, i.e., it has length equal to the codimension of ${\mathcal N}_{g,X}$
inside ${\mathbb A}^{n\choose 2}$. It is also
    Gorenstein if the resolution is also {\bf symmetric}, i.e., the
betti numbers are symmetric, see e.g., \cite[Chapter~21]{Eisenbud}
and specifically Corollary~21.16 contained therein.
In addition, if $S\subseteq T$ acts on $X$ we can speak of the
$S$-graded finite free resolution:
\begin{equation}
\label{eqn:resolution}
0\rightarrow F_N\rightarrow \cdots\rightarrow F_1\rightarrow F_0\rightarrow{\mathbb C}[{\bf z}]/I_{w,X}\rightarrow 0,
\end{equation}
If this resolution is short (respectively, also symmetric)
then ${\mathcal N}_{w,X}$ is Cohen-Macaulay (respectively Gorenstein).

\medskip
\noindent
(P.III) \emph{A decomposition of $X$:} From (\ref{eqn:Bruhat})
one inherits a decomposition
\[X=\coprod_{w} {\mathcal J}_{w,X} \mbox{ \ where \ } {\mathcal J}_{w,X}=(BwB/B)\cap X.\]
Recall that a function $f:X\to {\mathbb R}$
is {\bf upper-semicontinuous  on $X$} if for any $c\in {\mathbb R}$, the set
$\{p\in X|f(p)\geq c\}$ is closed. In good cases, one can compare points of ${\mathcal J}_{w,X}$.

\begin{lemma}
\label{lemma:goodcases}
Suppose $L$ is a linear algebraic group that acts on $X$.
\begin{itemize}
\item[(I)] If $L$ acts transitively on ${\mathcal J}_{w,X}$, there is a local
isomorphism between any two points in ${\mathcal J}_{w,X}$;
\item[(II)] If $wB\in {\overline{L\cdot gB}}$ for all $gB\in {\mathcal J}_{w,X}$
and $f:X\to {\mathbb R}$ is $L$-invariant and upper-semicontinuous on $X$, then
    this function is maximized on ${\mathcal J}_{w,X}$ at $wB$.
\end{itemize}
\end{lemma}
\begin{proof}
For (I), the elements of $L$ provide the isomorphisms. The claim (II) follows from the assumption of $L$-invariance and
the definition of upper-semicontinuity.
\end{proof}

In general, the ${\mathcal J}_{w,X}$ need not be smooth (see, e.g., the case of Richardson varieties in Section~7.2).
The reader may wish to compare \cite{Bialynicki-Birula}.

\medskip
\noindent
(P.IV) \emph{Localization of the equivariant $K$-class of $X$:}
Let $S$ be an $r$-dimensional (sub)torus of $T$ acting on $GL_n/B$.  For any such (sub)torus we have the standard identification
$\Lambda={\rm Hom}(S,{\mathbb C}^{\star}) \cong \mathbb Z^r$.
Let $\chi_1=e^{\lambda_1},\chi_2=e^{\lambda_2},\ldots,\chi_r=e^{\lambda_r}$
be the characters of $S$
corresponding to a choice of basis $\lambda_1,\ldots, \lambda_r$ of $\mathbb Z^r$.

Consider ${\bf M}={\mathbb C}[\bfz]/I_{w,X}$ as a ${\mathcal O}(wB_{-}B/B)\cong {\mathbb C}[{\bf z}]$-module.
When $g=w$, we may suppose that (\ref{eqn:resolution}) is a finitely $S$-graded resolution, where
each free module is of the form
\[F_i\cong {\mathbb C}[\bfz](-{b_{i1}})\oplus\cdots\oplus {\mathbb C}[\bfz](-{b_{ir}}).\]
Then one can match our choice of characters and the grading so that
\[k(\module )=\sum_{i=0}^{N}(-1)^i\chi_1^{b_{i1}}\chi_2^{b_{i2}}\cdots\chi_r^{b_{ir}},\]
is the $K$-polynomial of ${\bf M}$.
This polynomial always exists, and is independent of the choice of
free resolution \cite[Theorem 8.34]{Miller.Sturmfels}.

Let $K_S(GL_n/B)$ be the free abelian group generated by isomorphism classes of $S$-equivariant vector bundles on $GL_n/B$. (The smoothness of $GL_n/B$ implies that there is
a ${\rm Rep}(S)$-module isomorphism that identifies $K_S(GL_n/B)$
with the Grothendieck ($K$-homology) group $K^S(GL_n/B)$ of $S$-equivariant coherent sheaves
on $GL_n/B$.)

Suppose $wB\in (GL_n/B)^T\cap X$. Then $K_S(wB)$ is identified with representation ring of $S$, i.e., the group algebra
\[{\rm Rep}(S)={\mathbb Z}[\Lambda]=\bigoplus_{i=1}^{r}{\mathbb Z}\cdot
\chi_i.\]
$K_S(GL_n/B)$ has the structure of a ${\rm Rep}(S)$-module.
The map
$\rho:\{wB\}\hookrightarrow GL_n/B$
gives rise by pullback to the restriction map
$\rho^\star: K_S(GL_n/B) \rightarrow {\rm Rep}(S)$.

If $X\subseteq GL_n/B$ is an $S$-stable subscheme,
let $[{\mathcal O}_X]$
denote the class of its ($S$-equivariant) structure sheaf. If $wB\in X$,
we use common convention by writing the {\bf localization class}
\[[{\mathcal O}_X]|_{wB}=\rho^{\star}([{\mathcal O}_X])
\in {\rm Rep}(S)\cong K_S(wB).\]

 Then 
\begin{equation}
\label{eqn:Oequalsk}
[{\mathcal O}_X]|_{wB}=k(\module);
\end{equation}
see \cite[Proposition 8.23]{Miller.Sturmfels}.

In situations such as (P.II)(i) and (ii), one can compute $k({\bf M})$ without explicitly knowing the resolution.

Since (\ref{eqn:Bruhat}) is an $S$-invariant CW-decomposition of $GL_n/B$,
it follows that $GL_n/B$ is an equivariantly formal variety; see \cite[Theorem 14.1]{GKM}. This implies that to determine the class of an $S$-invariant $X\subseteq GL_n/B$ in $K_S(GL_n/B)$ it suffices to determine its localization at each $wB \in X\cap (GL_n/B)^T$. See \cite{KnuRos} and references therein.

\medskip
\noindent
(P.V) \emph{The (projectivized) tangent cone:}
Let $({\mathcal O}_{p,X}, {\mathfrak m}_p)$ denote the local ring of $p\in X$.
Then
\[{\rm gr}_{{\mathfrak m}_p}{\mathcal O}_{p,X}=\bigoplus_{i\geq 0}{\mathfrak m}_p^i/{\mathfrak m}_p^{i+1}\]
is the {\bf associated graded ring}.
The {\bf projectivized tangent cone} is ${\rm Proj}({\rm gr}_{{\mathfrak m}_p}{\mathcal O}_{p,X})$.
The {\bf Hilbert-Samuel multiplicity} of $p\in X$ is defined to be
${\rm mult}_{p}(X)={\rm deg}({\rm Proj}({\rm gr}_{{\mathfrak m}_p}{\mathcal O}_{p,X}))$.
This gives useful data about the singular structure of $p$. However,
more refined data is available from the {\bf $h$-polynomial}
$h_{p, X}(\uu)$, which is the numerator of the Hilbert series:
\[{\rm Hilb}({\rm gr}_{{\mathfrak m}_p}{\mathcal O}_{p,X},\uu)
=\frac{h_{p,X}(\uu)}{(1-\uu)^{\dim(X)}}.\]

One can compute these invariants of $X$ for $p=gB$ from the generators for $I_{g,X}$:
 Homogenize the generators $\{g_{k}\}$ by $t$. Now pick any term order $\prec$ that favors terms with largest degrees of $t$ first.
  Compute
a Gr\"{o}bner basis with respect to $\prec$ and set $t=1$ in each generator. The lowest degree
forms of the resulting polynomials define the ideal $I_{g,X}'\subset {\mathbb C}[\bfz]$, where
$I_{g,X}'$ (scheme-theoretically) cuts out ${\rm Spec}({\rm gr}_{{\mathfrak m}_p}{\mathcal O}_{p,X})$. See \cite[Proposition~15.28]{Eisenbud}.

Thus, $h_{gB,X}(\uu)$ is the numerator of the Hilbert series of
${\mathbb C}[\bfz]/I'_{g,X}$ (when expressed with denominator $(1-\uu)^{{\rm dim}(X)}$).
It is also true that the Hilbert-Samuel multiplicity is given by $h_{gB,X}(1)$; see, e.g.,
\cite[Theorem~5.4.15]{KreRob}.
Since $I_{g,X}'$ has the standard grading, one can check if it is Cohen-Macaulay
or Gorenstein by computing a minimal free resolution for ${\mathbb C}[{\bf z}]/I_{g,X}'$
and checking if it is short (respectively, also symmetric); cf. (P.II).

\smallskip
\begin{example}
\label{exa:petex}
Let $X={\rm Pet_n}$. The ideal $I_{w_P,\Pet}$ in Theorem~\ref{thm:main1} is
the patch ideal from (P.I). The
``typically easy'' determination of the equations defining $I_{w_P,\Pet}$ up to radical (P.II) is the content of the computations
(\ref{eqn:bigone}), (\ref{gformula}) and (\ref{eqn:alphaformula}) that pull back the definition of the Peterson variety along
$\pi:GL_n\to GL_n/B$. We will use (P.II)(i) to prove they actually define $I_{w_P,\Pet}$ as a radical ideal,
and hence $Y_{w_P,\Pet}\cong {\mathcal N}_{w_P,\Pet}$ (as schemes). The decomposition (\ref{eqn:thedecomposition})
agrees with (P.III); it just happens here that ${\mathcal J}_{w,\Pet}$
is a group orbit.

If $gB\in {\rm Pet}_n$ is any point then by (\ref{eqn:thedecomposition}) and (\ref{eqn:Jaffine}), we may suppose
$gB\in {\mathcal J}_{w_P,\Pet}$ where $g=bw_P$ for $b\in U_P$.
Re-centering so the origin is at $gB$ corresponds to the coordinate change
${z}_{i,j}\mapsto z_{i,j}+ \sum_{i<k} b_{i,k} \cdot z_{k, j}$,
where $b=(b_{ij})\in U_P$. Therefore, we may take
$I_{g,\Pet}\subseteq {\mathbb C}[{\bf z}]$ to be
generated by $\{g_{k,j}\}$ (after this affine linear substitution into each $g_{k,j}$).
Thus, Theorem~\ref{thm:main1} trivially extends to $I_{g,\Pet}$.

Theorem~\ref{thm:localize} is proved
via (P.IV). We develop Conjecture~\ref{conj:Petmain} using (P.V).\qed
\end{example}

\begin{example}
\label{exa:schub}
The use of patches to study Schubert varieties of $GL_n/B$
appears throughout the literature, see, e.g., \cite{KL, Knutson:patches} and the references therein. It being so ubiquitous an idea, we do not know a primary source. However, even for Schubert varieties, explicit study of the patch \emph{ideal} seems relatively recent; it appears at least
implicitly in \cite{Fulton} with more work in \cite{Knutson.Miller}.

Closely related is the use of the Kazhdan-Lusztig ideal, which is mostly the same as the patch ideal (P.I); thus
the themes of Section~2 apply. An explicit
Gr\"{o}bner basis was given in \cite{WY:KL}
(after the Gr\"{o}bner degeneration of \cite{Knutson:patches}, see also \cite{Knutson:frob}); cf. (P.II)(ii).
The decomposition of (P.III) is the Bruhat decomposition.
Combinatorial formulas for their localizations
(i.e., specializations of \emph{Grothendieck polynomials}) are geometrically explained via the point of view (P.IV).
Combinatorial analysis of the projectivized tangent cone in \cite{Li.Yong, Li.Yong2} illustrates (P.V). Also, see Section~7.2.
\qed
\end{example}

\begin{remark}
Suppose $Y\subseteq {\mathbb P}^n$ is any projective variety, then
let $\rho:GL_n/B\to {\mathbb P}^n$
be the map that projects
to the first linear subspace of a flag. This is a smooth morphism,
and in fact a locally trivial fibration. Thus, local
features of $Y$ can, \emph{in principle}, be  studied by considering
$X=\rho^{-1}(Y)\subseteq
GL_n/B$. \qed
\end{remark}

\begin{remark}
If $X\subseteq G/P$, then take $X\cap gB_{-}P/P$ to be an $X$ patch.
See work of A.~Fink-D.~Speyer \cite{Fink.Speyer} where the equivariant $K$-localizations of ${\overline{Tx}}$ in the Grassmannian
are connected to matroid invariants.

Separately, see thesis work of M.~Snider \cite{Snider} where the patch ideal of the \emph{positroid variety}
in the Grassmannian is computed and the scheme is related to certain Kazhdan-Lustzig varieties (essentially also patch ideals) of
the affine Grassmannian. See earlier, related work
of A.~Knutson-T.~Lam-D.~Speyer \cite{KLS}.

At least for the classical groups $G$, it is not difficult to extend the constructions (P.I)--(P.V) to attempt to understand
$X\subseteq G/B$. Briefly, one imposes additional quadric equations that account for the relevant bilinear form.
\qed
\end{remark}

\section{Further comparisons with earlier work on Peterson varieties}
D.~Peterson introduced and studied the strata
${\mathcal Y}_{w_P}={\rm Pet}_n\cap B_{-}w_{P^c} B/B$,
where $P^c$ is the parabolic subgroup for the nodes of the Dynkin diagram complementary to those
defining $P$. (One has ${\rm Pet}_n=   \coprod_{P} {\mathcal Y}_{w_P}$.)
Although these
strata are reducible (for examples of this in the Grassmannian case, see
\cite{Konnie:duke}), he announced that they are in fact reduced. He
furthermore gave a remarkable connection between ${\mathcal Y}_{w_P}$ and quantum cohomology of the partial
flag varieties $GL_n/P$. That is,
\begin{equation}
\label{eqn:iso}
{\mathbb C}[{\mathcal Y}_{w_P}]\cong QH^{\star}(GL_n/P);
\end{equation}
see \cite{Peterson}. B.~Kostant
\cite{kostant} had proved the case ${\mathcal Y}_{w_B}$ of this isomorphism.

J.~Tymoczko \cite{Tym07} gave a paving by affine spaces for regular nilpotent Hessenberg varieties. The cells in the case of $\Pet$ are
${\mathcal J}_{w_P,\Pet}={\rm Pet}_n\cap Bw_PB/B$;
see (\ref{eqn:Jaffine}) and Lemma~\ref{prop:Petjust}(IV).

In general ${\mathcal N}_{w_P,\Pet}$, ${\mathcal Y}_{w_P}$, and ${\mathcal J}_{w_P, \Pet}$ differ. In particular, the latter two do not give affine neighbourhoods of the $T$-fixed points of
${\rm Pet}_n$. However, ${\mathcal N}_{id, \Pet}={\mathcal Y}_{w_B}$.

K.~Rietsch \cite{reitsch}
proved Peterson's theorem (\ref{eqn:iso}) in type $A$, using
results of S.~Fomin-S.~Gelfand-A.~Postnikov \cite{FGP} and I.~Ciocan-Fontanine \cite{Ciocan}.
With a variation on the coordinates and equations we use for ${\mathcal N}_{g,\Pet}$
one can study ${\mathcal Y}_{w_P}$ for the classical
groups $G$, in connection to $QH^{\star}(G/P)$.
Here is a small sample:

\begin{example}
When $n=3$, the quantum cohomology ring is \[ QH^{\star}(GL_3/B)\cong \frac{{\mathbb C}[x_1,x_2,x_3,q_1,q_2]}{ \langle E_1^{(3)}, E_2^{(3)}, E_3^{(3)}\rangle} \] where
$E_1^{(3)} = x_1 +x_2+x_3, E_2^{(3)} = x_1x_2+x_1x_3+x_2x_3+q_1+q_2$ and $E_3^{(3)}= x_1x_2x_3+x_1q_2+x_3q_1$
are the \emph{quantum elementary symmetric polynomials}. This presentation
is an instance of a result of A.~Givental--B.~Kim and of I.~Ciocan-Fontanine; see, e.g., \cite[Equation~(1.2)]{FGP} and the associated references.
Setting
$x_3=-(x_1+x_2), \mbox{\ and $q_2 = -(q_1+x_1x_2+x_2x_3+x_1x_3)$}$,
we eliminate two variables to obtain
\[ QH^{\star}(GL_3/B) \cong \frac{{\mathbb C}[x_1,x_2,q_1]}{\langle x_1^3-2x_1q_1-x_2q_1 \rangle}.\]
Consider the surjection
$\psi: {\mathbb C}[Z^{(id)}]\cong {\mathbb C}[z_{21},z_{31},z_{32}] \rightarrow QH^{\star}(GL_3/B)$
defined by
\[z_{21} \mapsto  -( x_1+x_2 ), \ \ \  z_{31} \mapsto x_1x_2+q_1, \ \ \   z_{32} \mapsto  - x_1. \]

Observe
\begin{align*}  x_1^3-2x_1q_1-x_2q_1 & =  (-x_1-x_2)(x_1x_2+q_1)+ [(x_1x_2+q_1)-(-x_1-x_2)^2](-x_1) &\\
& = z_{21}z_{31} + (z_{31}-z_{21}^2)z_{32}. \end{align*}
By Theorem~\ref{thm:main1}, the
latter polynomial cuts out $Y_{123,{\rm Pet}_3}\cong {\mathcal N}_{123,{\rm Pet}_3}={\mathcal Y}_{w_B}$. It therefore follows that
$\psi$ induces ${\mathbb C}[{\mathcal Y}_{w_B}]\cong QH^{\star}(GL_3/B)$. This agrees with (\ref{eqn:iso}); cf.~\cite{reitsch}.\qed
\end{example}

\section{Proof of Theorem~\ref{thm:main1}}

\subsection{Preliminaries}

Let $U'\subset U$ be the group $U_P$ where $W_P=S_n$; see (\ref{eqn:BPdef}).

\begin{example} \label{example:BP}
In $GL_6$, $U'$ and $U_P$ for $W_P=S_3\times S_3$, respectively are:
\[U'= \left\{\begin{pmatrix}  1& q_1 & q_2  &q_3 & q_4& q_5\\ 0&1 & q_1 &q_2 & q_3& q_4\\ 0& 0& 1&q_1 &q_2 &q_3
\\0 &0 &0  & 1&q_1 &q_2 \\0&0 &0  &0 & 1&q_1 \\ 0&0&0&0&0& 1 \end{pmatrix}\right\} \hspace{1cm} U_P=\left\{
\begin{pmatrix}  1& q_1 & q_2  &0 & 0& 0\\ 0&1 & q_1 &0 & 0& 0\\ 0& 0& 1 &0 &0 &0 \\0 &0 &0  & 1&r_1 &r_2 \\0&0 &0  &0 & 1&r_1 \\ 0&0&0&0&0& 1 \end{pmatrix}\right\},\]
where $q_i, r_i\in {\mathbb C}$ for all $i$. \qed
\end{example}

Recall the one-dimensional torus
\[ S = \left \{ s(t)=\begin{pmatrix}  t &  & & \\  & t^2 & & \\ & &  \ddots & \\ & & & t^n \end{pmatrix} : t \in {\mathbb C}^{\star} \right \} \subset T.\]

We will need the following (known) facts.
 \begin{lemma}
 \label{prop:Petjust}
 \begin{itemize}
 \item[(I)] $S$ acts on ${\rm Pet}_n\subseteq GL_n/B$ by left multiplication.
 \item[(II)] Similarly $U'$ acts on ${\rm Pet}_n$, and ${\rm Pet}_n = \overline{U' \cdot w_0B}.$
     \item[(III)] ${\rm Pet}_n$ is irreducible.
   \item[(IV)] The two definitions of ${\mathcal J}_{w_P,\Pet}$ (cf. Section~1.2 and (P.I)) agree:
   \[{\mathcal J}_{w_P,\Pet} = U_P \cdot w_P B = \Pet\cap (Bw_P B/B).\]
     \item[(V)] If $gB\in {\mathcal J}_{w_P,\Pet}$, then
     $\lim_{t\to \infty} s(t)\cdot gB=w_P B$.
    \item[(VI)] \label{lemma:whatarefixedpoints}
$wB \in {\rm Pet}_n$ if and only if $w=w_P$, where $w_P$ is the maximal Coxeter length element of $W_P=S_{i_1}\times S_{i_2}\times
\cdots \times S_{i_k}$. Explicitly,
\[w_P=i_1 \ i_1 -1 \cdots 3 \ 2 \ 1 \ i_1+i_2 \ i_1 +i_2 -1 \ \cdots i_1+1 \ \cdots \ \cdots \ n \ n-1 \ \cdots i_1+\cdots +i_{k-1}+1.\]
 \end{itemize}
 \end{lemma}

\begin{example}
Let $n=8$ and $W_P=S_3\times S_2\times S_3$. Then the permutation
matrix for $w_P$ and the generic matrix $Z^{(w_P)}$ in $w_PU_{-}$ are given by
\[w_P=\left(\begin{matrix}
0 & 0 & 1 & 0 & 0 & 0 & 0 & 0 \\
0 & 1 & 0 & 0 & 0 & 0 & 0 & 0 \\
1 & 0 & 0 & 0 & 0 & 0 & 0 & 0 \\
0 & 0 & 0 & 0 & 1 & 0 & 0 & 0 \\
0 & 0 & 0 & 1 & 0 & 0 & 0 & 0 \\
0 & 0 & 0 & 0 & 0 & 0 & 0 & 1 \\
0 & 0 & 0 & 0 & 0 & 0 & 1 & 0 \\
0 & 0 & 0 & 0 & 0 & 1 & 0 & 0
\end{matrix}\right) \mbox{ \  \ and \ \ }
Z^{(w_P)}=\left(\begin{matrix}
z_{11} & z_{12} & 1 & 0 & 0 & 0 & 0 & 0 \\
z_{21} & 1 & 0 & 0 & 0 & 0 & 0 & 0 \\
1 & 0 & 0 & 0 & 0 & 0 & 0 & 0 \\
z_{41} & z_{42} & z_{43} & z_{44} & 1 & 0 & 0 & 0 \\
z_{51} & z_{52} & z_{53} & 1 & 0 & 0 & 0 & 0 \\
z_{61} & z_{62} & z_{63} & z_{64} & z_{65} & z_{66} & z_{67} & 1 \\
z_{71} & z_{72} & z_{73} & z_{74} & z_{75} & z_{76} & 1 & 0 \\
z_{81} & z_{82} & z_{83} & z_{84} & z_{85} & 1 & 0 & 0
\end{matrix}\right). \qed\]
\end{example}

\noindent
\emph{Proof of Lemma~\ref{prop:Petjust}:}
For (I), the claim follows from the definition of ${\rm Pet}_n$ since
$(s(t))^{-1} N s(t)=tN$ for $s(t) \in S$.
Similarly, the first sentence of (II) holds since $U'$ centralizes $N$.

(III) was shown in \cite[Theorem~6]{kostant}. One knows
$\dim \Pet =n-1$. Clearly, ${\rm dim}(U' \cdot w_0B)=n-1$, so the (irreducible) closure of this group orbit is all of $\Pet$; this is~(II).

For (IV), flags in $Bw_{P}B$ may be identified with the block matrices
$\matrx =\matrx^{(1)}\oplus \matrx^{(2)}\oplus\cdots\oplus \matrx^{(k)} \in GL_n$
where $\matrx^{(j)}$ is the matrix with $1$'s on the antidiagonal, $0$'s below, and free entries above.
The condition that the given flag is in $\Pet$ implies that there are equal entries on each super-antidiagonal of
each $K^{(i)}$. The set of all such $\matrx$ is equal to $U_P w_{P}B$.

(V) holds by the second equality of (IV) and the fact that the claim is
true of any point of $Bw_P B/B$, by elementary considerations.

Lastly for (VI), by definition,
$F_{\bullet}=wB \in {\rm Pet}_n$ if and only if $N\cdot F_i \subseteq F_{i+1}$
for $1\leq i\leq n-2$.
The latter happens if and only if
$N\cdot {\vec e_{w(i)}}={\vec e_{w(i)-1}}\in F_{i+1}$
whenever $w(i)\neq 1$ (when $w(i)=1$,
$N{\vec e_{w(i)}}={\vec 0}\in F_{i+1}$ is automatic). Here ${\vec e_j}$ is the $j$-th standard basis vector.
That $w=w_{P}$ for some $P$ then follows
from this, by inducting on the position of the ``$1$'' in the permutation
matrix of $w$, from left to right.\qed

\subsection{Proof of Theorem~\ref{thm:main1}:} We refer the reader to the discussion of (P.I), (P.II) and Example~\ref{exa:petex}
in Section~2:
Let $gB\in {\mathcal N}_{w_P,\Pet}$. Since
${\mathcal N}_{w_P,\Pet}\cong \pi^{-1}(\Pet)\cap \sigma_{w_P}(w_P B_{-}B/B)\cong
\pi^{-1}(\Pet)\cap  w_P U_{-}$ we may assume $g$ is the unique element of $w_PU_{-}$ sent to
$gB$ under $\pi:GL_n\to GL_n/B$. Since $gB\in \Pet$, it must therefore be true that the
$j$-th column of $g$ is contained in the span of the first $j+1$ columns. Therefore, by construction, the
entries of $g$ (considered as coordinates in $\CC^{n \choose 2}$) must satisfy the defining equations of $I_{w_P,\Pet}$.
Hence $Y_{w_P,\Pet}\supseteq {\mathcal N}_{w_P,\Pet}$. The other containment is the same argument
run in reverse. Thus $Y_{w_P,\Pet}={\mathcal N}_{w_P,\Pet}$ (set-theoretically), where the latter is viewed as a subvariety of ${\mathbb C}^{n\choose 2}$.

The dimension of ${\rm Pet}_n$ is $n-1$; the same is true of ${\rm Pet}_n\cap w_P\cdot B_{-}B/B$. Thus the codimension of ${\mathcal N}_{w_P,\Pet}$ and $Y_{w_P,\Pet}$ in ${\mathbb C}^{n\choose 2}$ (as identified with
$\sigma_{w_P}(w_P\cdot B_{-}B/B)$) is ${n-1\choose 2}$. This is the number of generators of $I_{w_P, \Pet}$.
Hence $Y_{w_P,\Pet}$ is a complete intersection (and Cohen-Macaulay) \cite[Prop~18.13]{Eisenbud}.

We now show that $Y_{w_P,\Pet}$ is reduced; cf. (P.II).
Since ${\rm Pet}_n$ is irreducible, ${\mathcal N}_{w_P,\Pet}=Y_{w_P,\Pet}$ (set-theoretic equality) is irreducible. Thus,
by (P.II)(i),
it suffices to exhibit a smooth point,
which we do with the Jacobian criterion.
To achieve this, we will use the following fact, which is immediate by an induction using (\ref{eqn:alphaformula}):

\begin{lemma}
\label{lemma:alphaleft}
Fix $1\leq s\leq n-2$ and let $1\leq q\leq s+1$. Then $\alpha_{s,q}$
is a polynomial in $\{z_{t,r}: r\leq s\}$.
\end{lemma}

Consider
the summand $\alpha_{j,j+1}z_{k,j+1}$ of the generator $g_{k,j}$ of $I_{w_P,\Pet}$; this is the unique summand in $g_{k,j}$
involving $z_{k,j+1}$ by Lemma~\ref{lemma:alphaleft}.
The assumption that $g_{k,j}$ is a generator implies $z_{k,j+1}$ is unspecialized. Lemma~\ref{lemma:alphaleft} additionally implies that
$z_{k,j+1}$ does not appear in ${\alpha}_{j,j+1}$. Thus,
\[\frac{\partial g_{k,j}}{\partial z_{k,j+1}}=\frac{\partial(\alpha_{j,j+1}z_{k,j+1})}{\partial z_{k,j+1}}=\alpha_{j,j+1}({\bf z}).\]
Similarly, the variable $z_{k,j+1}$ does not appear in any
$g_{k',j}$ for $k'\neq k$, or $g_{k',j'}$ for  $j'<j$.

\begin{claim}
There exists a point
\begin{equation}
\label{eqn:supposepoint}
p \in Y_{w_P,\Pet} \mbox{\ such that $\alpha_{j,j+1}( p )\neq 0$ for all $1\leq j\leq n-2$.}
\end{equation}
\end{claim}
\begin{proof}
Consider
$H_j=\{F_{\bullet}\in GL_n/B: N\cdot F_i\subset F_{i+1} \mbox{\ if $i\neq j$ and \ } N\cdot F_j\subseteq F_{j}\}$.
In fact $H_j$
is a regular nilpotent Hessenberg variety, see Section~7.1. It follows from a dimension formula of E.~Sommers-J.~Tymoczko \cite[Theorem~10.2]{SomTym06}, which is explicitly stated in \cite[Section~9.1]{Tym06a},  that
$\dim H_j=n-2$. Notice that
${\mathcal N}_{w_P,H_j}=Z(\alpha_{j,j+1})\cap Y_{w_P,{\rm Pet}_n}$ (set-theoretically).

Suppose
$\Gamma({\bf z})=\prod_{1\leq j\leq n-2}\alpha_{j,j+1}({\bf z})$
vanishes identically on $Y_{w_P,\Pet}$. Then this means (set-theoretically) that
$Y_{w_P,\Pet}\subseteq \bigcup_{j=1}^{n-2} {\mathcal N}_{w_P,H_j}$.
Hence $Y_{w_P,\Pet}$ is at most $n-2$ dimensional, a contradiction. Thus, one
can choose a point $p\in Y_{w_P,\Pet}$ not on the hypersurface defined by $\Gamma({\bf z})$. This point satisfies (\ref{eqn:supposepoint}).
\end{proof}

Next, arrange the rows of the
Jacobian $\left[\frac{\partial g_{k,j}}{\partial \zz}\right]$ associated to generators $g_{k,j}$ by (say) favoring
smaller $j$ first and breaking ties by favoring smaller $k$. This manifests (via echelon form) that the Jacobian, evaluated at $p$,
has full rank (having established ${n-1 \choose 2}$ linearly independent columns)
and thus $p$ is a smooth point of $Y_{w_P,\Pet}$.\qed

\section{Proof of Theorem~\ref{thm:singlocus}}

The proof is in two parts. First, we need that any $gB=bw_P B\in {\mathcal J}_{w_P,\Pet}$ is singular when $w_P$ is not on the
asserted ``smooth list'' from (II). Next, we prove that for the three $w_P$ on that list, the points $w_P B$ are smooth.
Since the point $w_P B$ is the most singular in ${\mathcal J}_{w_P,\Pet}$
(by Lemma~\ref{lemma:goodcases}(II) and Lemma~\ref{prop:Petjust}(V)),
all points in that cell are smooth. This establishes the condition (I) and the equivalence of (I) with (II).

For the first part, recall that if $v\in S_n$ and $w\in S_{N}$ for $n\leq N$ then  $w$ {\bf contains the pattern} $v$ if there exists
a choice of {\bf embedding indices}
$1\leq \phi_1<\phi_2<\ldots<\phi_n\leq N$
such that $w(\phi_1),w(\phi_2),\ldots,w(\phi_n)$
are in the same relative order as $v(1),\ldots,v(n)$.

 It is straightforward to verify the equivalence (II)$\Leftrightarrow$(III)
from this definition. We say that $w$ {\bf avoids} $v$ if no such indices exist.

The following two lemmas are clear from the Lemma~\ref{lemma:whatarefixedpoints}(VI):

\begin{lemma} \label{phi'}
Suppose $w_P B \in {\rm Pet}_n$. If $w_P$ contains the pattern $123$ then one can choose
the embedding indices $\phi_1<\phi_2<\phi_3$ so that $w_P(\phi_1)=1$ and $\phi_2=\phi_1+1$.
\end{lemma}

\begin{lemma} \label{2143}
Suppose that $w_P B\in {\rm Pet}_n$ and avoids $123$.
If $w_P$ contains the pattern $2143$ then one can choose the embedding
indices $\phi_1<\phi_2<\phi_3<\phi_4$
to satisfy $\phi_3=\phi_2+1$.
\end{lemma}



Recall the notation and definitions preceding the definition of $I_{w_P,\Pet}$ in Section~1.
Although we need to prove that all points $bw_P B$ are singular, strictly speaking we only show
$w_P B$ is singular, for simplicity of notation. However, in view of the discussion in Example~\ref{exa:petex},
 the general case is a trivial extension. Specifically,
Lemma~\ref{lemma:nonzeroconstantterms}, the three claims
that follow it, and their proofs, are the exactly same in the general case, where we have
made the substitution $z_{ij}\mapsto z_{ij}+\sum_{i<k}b_{ik}z_{k,j}$ from
Example~\ref{exa:petex}. Thus, we leave the details to the interested reader.

\begin{lemma} \label{lemma:nonzeroconstantterms}
Suppose $\phi$ satisfies $w_P(\phi)=1$. Then none of the $\alpha_{\phi,k} ( {\bf z} )$ for $1\leq k\leq \phi+1$ contains a nonzero constant term.
\end{lemma}
\begin{proof}
Since $w_P(\phi)=1$, $N \cdot Z_{\phi}$ is a column vector that only involves zeros and the unspecialized variables $z_{t,\phi}$. Hence
$\alpha_{\phi,1}$ satisfies the stated conclusion. The claim follows for all the
$\alpha_{\phi,k}$ by (\ref{eqn:alphaformula}) and induction on~$k$.
\end{proof}

Suppose $w_P$ contains $123$. Let $\phi_1<\phi_2<\phi_3$
be the embedding indices from Lemma~\ref{phi'}: $1=w_P(\phi_1)<w_P(\phi_2)<w_P(\phi_3)$ with $\phi_2=\phi_1+1$.

\begin{claim}
\label{claim:onedefine}
$g_{n,\phi_1}$
is one of the defining generators of $I_{w_P,\Pet}$.
\end{claim}
\begin{proof}
From the definitions the assertion is clearly true provided
\begin{equation}
\label{eqn:provided}
n\neq w_P(\ell)  \mbox{\ for $1\leq\ell\leq \phi_1+1=\phi_2$.}
\end{equation}
Suppose (\ref{eqn:provided}) does not hold. Since $w_P(\phi_1)=1$, this bad $\ell$ is not equal to $\phi_1$.
Also, $\ell\neq \phi_1+1(=\phi_2)$ since this violates the assumptions about
$\phi_1,\phi_2,\phi_3$.
Hence $1\leq \ell<\phi_1$. This combined with $w_P(\phi_1)=1$ and the form
of $w_P$ given by Lemma~\ref{lemma:whatarefixedpoints}(VI) implies $w_P=w_0$. However, $w_0$
does not contain $123$, a contradiction. Thus the claim holds.
\end{proof}

    Since by Lemma~\ref{lemma:nonzeroconstantterms} we know
$\alpha_{\phi_1,k}( {\bfz})$ has no nonzero constant terms and since
\begin{eqnarray}\nonumber
g_{n,\phi_1}(\bfz) & = & \alpha_{\phi_1,1}(\bfz)  {z}_{n,1}+ \cdots + \alpha_{\phi_1,\phi_1+1}(\bfz)
{z}_{n,\phi_1+1}
\end{eqnarray}
the minimal degree of any term of $g_{n,\phi_1} (\bfz)$ is two.
    Thus the row of the Jacobian $\left[\frac{\partial g_{k,j}(\bfz)}{\partial \zz}\right]$ (evaluated at the origin)
    associated to $g_{n,\phi_1}({\bf {z}})$ consists only of zeros. As the number of rows in Jacobian equals the ${\rm codim} \ Y_{w_P,\Pet}$, the rank at that point is strictly smaller than the codimension of $Y_{w_P,\Pet}={\mathcal N}_{w_P,\Pet}$ in ${\mathbb C}^{n\choose 2}$.
    Hence by the
    Jacobian criterion, the origin is not a smooth point. As the origin corresponds to $w_PB$ in our coordinates, $w_P B$ is singular in ${\rm Pet}_n$.

    Now suppose $w_P$ contains $2143$. In view of our analysis above, we may assume $w_P$ avoids $123$.
    Let $\phi_1<\phi_2<\phi_3<\phi_4$ be the embedding indices
from Lemma~\ref{2143}. Since $w_P$ avoids $123$, its form, as stated in Lemma~\ref{lemma:whatarefixedpoints}(VI)
satisfies $k\leq 2$. The containment of $2143$ means $k=2$ and $i_1,i_2\geq 2$. Hence we can
actually assume $w_P(\phi_1)=2$ and $w_P(\phi_2)=1$.

Using these assumptions, and an argument such as for (\ref{claim:onedefine}) it is straightforward that

\begin{claim}
$g_{n,\phi_1}({\bfz})$ and $g_{n-1,\phi_2}({\bfz})$ are defining generators of $I_{w_P,\Pet}$.
\end{claim}

\begin{claim}
The Jacobian $\left[\frac{\partial g_{k,j}(\bfz)}{\partial \zz}\right]$, when evaluated at the origin ${\bf 0}$, has the property that
the row corresponding to $g_{n,\phi_1}({\bfz})$
is a scalar multiple of the row corresponding to $g_{n-1,\phi_2}(\bfz)$.
\end{claim}
\begin{proof}
Observe
that $\alpha_{\phi_1,\phi_1+1} (\bfz) (=\alpha_{\phi_1,\phi_2} (\bfz))$
contains the constant term $1$ (coming from the $1$ at the top of $N{Z}_{\phi_1}$).
In addition, all entries in $N{Z}_{\phi_1}$, other than the top one, do not contain a constant term. By induction on $k\geq 1$,
$\alpha_{\phi_1,k}(\bfz)$ does not contain a constant term
for $1 \leq k \leq \phi_1$. Thus $g_{n,\phi_1}(\bfz)$ contains one linear term (and terms of higher degree), which is $z_{n,\phi_2}$ (from the lefthand side of (\ref{gformula})).

Since $w_P(\phi_2)=1$, Lemma \ref{lemma:nonzeroconstantterms} states that all of the $\alpha_{\phi_2,k}(\bfz)$ do not have
constant terms.  Hence $g_{n-1,\phi_2}(\bfz)$ contains a single
linear term (and terms of higher degree), i.e., the $z_{n,\phi_2}$ term coming from the right hand side of (\ref{gformula}).  Thus, in the column of the Jacobian (evaluated at the origin) associated to $z_{n,\phi_2}$ in rows $g_{n,\phi_1} (\bfz)$ and $g_{n-1,\phi_2}(\bfz)$ the entry is $\pm 1$ and $\mp 1$ respectively. In other columns, the entries in those two rows are zero.
\end{proof}

We have thus completed the argument for the first part of the proof.

We now turn to the second half of the proof, i.e., checking smoothness at the three
points indicated by the list in (II). As mentioned at the beginning of this section, smoothness
at these points implies smoothness at all other points of their respective strata.

\noindent
\emph{Case $w_P=w_0$:} If $w_0B$ is singular, then so is every point of $U'w_0 B$ by Lemma~\ref{prop:Petjust}(II) and Lemma~\ref{lemma:goodcases}(I)
combined. However, as ${\rm Pet}_n$ is
the closure of this orbit, that would imply every point of ${\rm Pet}_n$ is singular, which is impossible.

\smallskip
\noindent
\emph{Case $w_P = n-1 \ n-2 \ \cdots \ 2 \ 1 \ n$:} Fix a total order $<$ on the generators $\{ g_{k,j} \}$ by
\begin{equation}
\label{eqn:theorder}
(k',j') < (k,j) \mbox{\ if $j'<j$ and \ } (k',j) < (k,j) \mbox{\ if $k'<k$.}
\end{equation}

For this $w_P$,
 \begin{equation}
 \label{eqn:alphaspecial}
 \alpha_{j,j+1} = 1 - \left(\sum_{\ell =1}^{j} \alpha_{j,\ell } \ z_{w(j+1),\ell }\right)
 \end{equation}
  for all $1 \leq j \leq n-2$; cf.~(\ref{eqn:alphaformula}).  Hence, the term $z_{k,j+1}$ appears as a linear term in $g_{k,j}$.
  It does not appear in any $g_{k',j'}$ where $(k',j')<(k,j)$: if $j'<j$ then this variable appears nowhere in (\ref{eqn:bigone});
  if $j=j'$ and $k'\neq k$ no $\alpha_{j,\ell }$ involves this variable, nor does $NZ_{j}$, so neither can $g_{k',j}$.
  Therefore,
\begin{equation}
  \label{eqn:othercols}
  \frac{\partial g_{k,j}}{\partial z_{k,j+1}}  \Bigg |_{\bfz = {\bf 0}} = 1  \mbox{ \ and \ }  \frac{\partial g_{k',j'}}{\partial z_{k,j+1}}   \Bigg |_{\bfz = {\bf 0}}=0 \mbox{ \ for $(k',j')<(k,j)$.}
\end{equation}
We thus conclude the rows of the Jacobian (evaluated at the origin) are linearly independent.  Hence the point $w_PB$ is smooth.

\smallskip
\noindent
\emph{Case $w_P = 1,n,n-1, \ldots,2$:}  Isolating the case $j=1$ first, Lemma \ref{lemma:nonzeroconstantterms} implies that $\alpha_{1,1}$ and $\alpha_{1,2}$ do not contain any constant terms.  Thus, the only linear term in each $g_{k,1}$ is the
$z_{k+1,1}$ that comes from $NZ_{1}$ in (\ref{eqn:bigone}).  Therefore,
\begin{equation}\label{column1}  \frac{\partial g_{k,1}}{\partial z_{k+1,1}}  \Bigg |_{\bfz = {\bf 0}} = 1  \mbox{ \ and \ }  \frac{\partial g_{k',1}}{\partial z_{k+1,1}}   \Bigg |_{\bfz = {\bf 0}}=0 \mbox{ \ for  $k'\neq k$}.\end{equation}

For $2 \leq j \leq n-2$, (\ref{eqn:alphaspecial}) holds.
This means each $g_{k,j}$ with $2 \leq j \leq n-2$ contains $z_{k,j+1}$ as a linear term.
As with the case $w_P=n-1 \ n-2 \ \cdots \ 2 \ 1 \ n$,
$z_{k,j+1}$ does not appear in any $g_{k',j'}(\bfz)$ where $(k',j')<(k,j)$ in the order (\ref{eqn:theorder}).
Combining this with (\ref{column1}) we see (\ref{eqn:othercols}) holds for $1\leq j\leq n-2$. Again,
the columns of the Jacobian (evaluated at the origin) are linearly independent, and thus $w_PB$ is smooth. \qed

\section{$K$-polynomials, equivariant localization and $h$-polynomials}

\subsection{Proof of Theorem~\ref{thm:localize}}
We will in fact prove:
\begin{equation}
\label{eqn:infact}
[{\mathcal O}_{\rm Pet_n}]_S|_{w_P B}=k_{w_P, \Pet}(\chi)=\prod_{j=1}^{n-2}\prod_{\ell \neq w_P (i), 1\leq i\leq j+1}\left ( 1-\chi^{\ell+1-w_P(j)} \right ) \in K_S(w_P B)\cong {\mathbb Z}[\chi^{\pm 1}],
\end{equation}
where $k_{w_P,\Pet}(\chi)$ is the $K$-polynomial of ${\mathbb C}[{\bf z}]/I_{w_P,\Pet}$ with respect the ${\mathbb Z}$-graded coarsening of (\ref{eqn:thegrading}), associated to $S\subset T$. This coarsening assigns to the variable $z_{ij}$ the degree
\begin{equation}
\label{eqn:coarse}
{\rm deg}(z_{ij})=i-w_P(j).
\end{equation}
The first equality of (\ref{eqn:infact}) holds by the discussion of (P.IV) that concludes at (\ref{eqn:Oequalsk}).

We use a standard fact: if
$I=\langle g_1,\ldots,g_N\rangle\subset {\mathbb C}[x_1,\ldots,x_M]$ is a ${\mathbb Z}$-graded ideal that is
a complete intersection, then its $K$-polynomial equals
\[k({\mathbb C}[x_1,\ldots,x_M]/I; \chi)=\prod_{i=1}^N 1-\chi^{{\rm deg}(g_i)}.\]
This is an easy consequence of, e.g., \cite[Exercise~8.12]{Miller.Sturmfels}.

By Theorem~\ref{thm:main1}, we can apply the above fact to the generators
$\{g_{ \ell ,j}\}$. It remains to determine the degree of the generator $g_{\ell,j}$ for $1\leq j\leq n-2$ and
$ \ell \neq w_P(t)$ for $1\leq t\leq j+1$.
Since $g_{\ell,j}$ is
homogeneous with respect to the grading (\ref{eqn:coarse}), we need only compute the degree of any term of $g_{\ell,j}$
and show it is equal to $\ell+1-w_P(j)$.

Since $g_{\ell,j}$ is a generator, the definitions imply the summand $\alpha_{j,j+1}z_{\ell,j+1}$ appears.
If $w_P(j+1)<n$ then the desired degree is easily established since the unspecialized variable $z_{w_P(j+1)+1,j}$ appears in $\alpha_{j,j+1}$ by
(\ref{eqn:alphaformula}). Only slightly harder is the case $w_P(j+1)=n$ where this variable does not exist.
Here, by (\ref{eqn:alphaformula}), $\alpha_{j,j+1}$ contains the summand $\alpha_{j,1}z_{n,1}$. However $\alpha_{j,1}=z_{w_P(1)+1,j}$.
Therefore we conclude $\alpha_{j,j+1} z_{\ell,j+1}$ contains the term $z_{w_P(1)+1,j}z_{n,1}z_{\ell,j+1}$ which has
degree
$(w_P(1)+1-w_P(j))+(w_P(n)-w_P(1))+(\ell-w_P(j+1))=\ell+1-w_P(j)$.
The second equality of (\ref{eqn:infact}) is thus deduced. \qed

\subsection{The (projectivized) tangent cone}
Theorem~\ref{thm:singlocus} asserts that for ${\rm Pet}_n$,
the property of being singular is constant on each
${\mathcal J}_{w_P,\Pet}$.
Lemma~\ref{lemma:goodcases}(II) and Lemma~\ref{prop:Petjust}(V) combined imply Hilbert-Samuel multiplicity is
maximized on ${\mathcal J}_{w_P,\Pet}$ at $w_P B$.
\begin{question}
Is ${\rm mult}_{g,\Pet}$, or $h_{g,\Pet}(\uu)$, constant on ${\mathcal J}_{w_P,\Pet}$?
\end{question}

Let us restrict attention to the $h$-polynomials $h_{w_P,\Pet}(\eta)$.
These are tabulated for $n=4$ in Table~\ref{table:KH} below, along with the $K$-polynomials (which were independently computed using {\tt SAGE}, and agree with Theorem~\ref{thm:localize}).
We state the following easy facts:

\begin{lemma}
\label{lemma:BSsemi}
\begin{itemize}
\item[(I)] There is a local isomorphism between any two points of
the orbit $(U'\rtimes S)\cdot gB$.
\item[(II)] $(id)B \in {\overline{(U'\rtimes S)\cdot w_P B}}$.
\item[(III)] If $f:\Pet\to {\mathbb R}$ is any $U'\rtimes S$-invariant,
upper-semicontinuous function on $\Pet$, then it is maximized at $(id)B$.
\end{itemize}
\end{lemma}
\begin{proof}
(I) is Lemma~\ref{prop:Petjust} parts (I) and (II) combined with Lemma~\ref{lemma:goodcases}(I).
In general, if $gB\in BwB/B$ and $gB\neq wB$,
then $\lim_{t \rightarrow 0} s(t) \cdot gB = uB$ with $u< w$. (II) just applies this.
Finally, (III) holds by part (II),
together with the hypotheses on $f$.
\end{proof}

\begin{example}
 In general, the $U'\rtimes S$ action is not transitive on ${\mathcal J}_{w_P,\Pet}$.
 Consider the $4$-dimensional ${\mathcal J}_{w_P,{\rm Pet}_6}$ where $P=S_{3}\times S_3$.  One checks the $3$-dimensional subgroup $U'' \subset U'$,
obtained by setting $q_1 = q_2 = 0$ in the $U'$ of Example~\ref{example:BP},
fixes any $gB$ in $J_{w_P,{\rm Pet}_6}$. Thus,
any $U' \rtimes S$ orbit through $w_PB$
has dimension at most $3$ and hence cannot
cover ${\mathcal J}_{w_P,{\rm Pet}_6}$. \qed
\end{example}

Lemma~\ref{lemma:BSsemi}(III) says that $(id)B$ is the ``most singular''
on $\Pet$. This, and the connection of ${\mathcal N}_{id,\Pet}$ to $QH^{\star}(GL_n/B)$ (see Section~3) motivates this case of Problem~\ref{prob:B}:
\begin{problem}
Find a combinatorial formula for either ${\rm mult}_{id}(\Pet)$ and/or
$h_{id,\Pet}(\uu)$.
\end{problem}

The sequence of multiplicities at this point begins:
\[{\rm mult}_{123.B}({\rm Pet}_3)\!=\!2, \    {\rm mult}_{1234.B}({\rm Pet}_4)\!=\!4, \
{\rm mult}_{12345.B}({\rm Pet}_5)\!=\!12, \ {\rm mult}_{123456.B}({\rm Pet}_6)=38, \ \ldots\]
In addition, the sequence of $h$-polynomials begins:
\[h_{123.B, {\rm Pet}_3}(\uu)=1+\uu, \
h_{1234.B, {\rm Pet}_4}(\uu)=1+2\uu+\uu^2, \
h_{12345.B, {\rm Pet}_5}(\uu)=1+3\uu+4\uu^2+3\uu^3+\uu^4, \]
\[h_{123456.B, {\rm Pet}_6}(\uu)=1+4\uu+8\uu^2+10\uu^3+8\uu^4+4\uu^5+2\uu^6+\uu^7,\ldots.\]

\begin{question}
Are the coefficients of $h_{g,\Pet}(\uu)$ unimodal, or more strongly, log-concave?
\end{question}
While unimodality might be true of $h$-polynomials of Schubert varieties, log-concavity is known not to be true \cite[Example~2.2]{Li.Yong2}. These conjectures are not known to follow for any
    obvious reasons, even if we knew ${\rm gr}_{{\mathfrak m}_{gB}}{\mathcal O}_{gB,{\rm Pet}_n}$ to be Cohen-Macaulay.

The following example contrasts with Conjecture~\ref{conj:reduced} concerning Richardson varieties:

\begin{example}
\label{exa:nonreduced}
${\rm Spec}({\rm gr}_{{\mathfrak m}_{gB}}{\mathcal O}_{gB,{\rm Pet}_n})$
need not be reduced. This occurs for $2134\cdot B\in {\rm Pet}_4$ where the ideal
defining the tangent cone is
$I_{2134\cdot B,{\rm Pet}_4}'=\langle z_{32}-z_{41}, z_{42}, z_{41}^2\rangle$.
\end{example}

\begin{table} \begin{tabular}{ |l | c | l | } \hline $w_P$ &
$k_{w_P,\Pet}(\chi)$ & $h_{w_P, \Pet}(\uu)$ \\ \hline $1234$ & $(1-\chi^3)^2(1-\chi^4)$ &
$1+2\uu+\uu^2$ \\ \hline $2134$ & $(1-\chi^2) (1-\chi^3)(1-\chi^4) $& $1+\uu$ \\
\hline
$1324$ & $(1-\chi^2)^2(1-\chi^4)$ & $1+\uu$\\ \hline
$1243$ & $(1-\chi^2)(1-\chi^3)(1-\chi^4)$ &$1+\uu$ \\ \hline
$2143$ &$(1-\chi^3)^2(1-\chi^2)$ & $1+\uu$\\ \hline
$3214$ &$(1-\chi^{-1})(1-\chi^2)(1-\chi^3) $& $1$\\ \hline
$1432$ &$(1-\chi^{-1})(1-\chi^2)(1-\chi^3)$ & $1$\\ \hline
$4321$ & $(1-\chi^{-1})^2(1-\chi^{-2})$ & $1 $\\ \hline
\end{tabular}
\caption{\label{table:KH} $K$-polynomials and $h$-polynomials at $w_PB \in {\rm Pet}_4$}
\end{table}

\begin{conjecture}
\label{conj:Petmain}
${\rm gr}_{{\mathfrak m}_{gB}}{\mathcal O}_{gB,{\rm Pet}_n}$ is Cohen-Macaulay
and thus $h_{g,\Pet}(\uu)\in {\mathbb N}[\uu]$. The function $f:\Pet\to {\mathbb R}$ defined by $gB\mapsto$``coefficient of $\uu^i$ in $h_{g,\Pet}(\uu)$'', is upper-semicontinuous on $\Pet$.
\end{conjecture}

M.~Kummini pointed out that the asymmetry of the coefficients of
$h_{123456.B,{\rm Pet}_6}(\eta)$ computed above implies that
${\rm gr}_{{\mathfrak m}_{123456.B}}{\mathcal O}_{123456.B,{\rm Pet}_6}$ is not Gorenstein.

For a point $p\in X$,
the Cohen-Macaulayness of the local ring ${\mathcal O}_{p,X}$ is necessary but not sufficient for
${\rm gr}_{{\mathfrak m}_{p}}{\mathcal O}_{p,{\rm Pet}_n}$ to be Cohen-Macaulay.
That Cohen-Macaulayness implies the positivity of the $h$-polynomial is \cite[Corollary~4.1.10]{Bruns.Herzog}.
While Hilbert-Samuel multiplicity is upper-semicontinuous, the coefficients of the $h$-polynomial need not be.

Using (P.V) we verified the Cohen-Macaulayness claim of Conjecture~\ref{conj:Petmain} for $n\leq 4$ and all but
two cases of $n\leq 5$, as well as in many other instances. Table~\ref{table:KH} is also consistent with the upper-semicontinuity claim.

\section{Analysis of additional torus invariant varieties}

Conjecture~\ref{conj:Petmain} is an analogue of a conjecture stated for Schubert varieties in
 \cite[Section~1.2]{Li.Yong2}. In regards to the second meta-question of Section~1.1,
we briefly explore the prevalence of the properties of Conjecture~\ref{conj:Petmain}, by examining some other $X\subseteq GL_n/B$.

\subsection{Hessenberg varieties and Springer fibers}

Fix $\nilp \in M_{n\times n}$ and ${\mathfrak h}:{\mathbb N}\to {\mathbb N}$.
Define
${\rm Hess}( \nilp ,{\mathfrak h})=\{F_{\bullet}\in GL_n/B : \ \nilp\cdot F_i\subseteq F_{ {\mathfrak h} (i) } \}$.
These were studied by \cite{DPS}. The definition we use is stated in \cite{Tym06a}.
Two families of Hessenberg varieties arise from the cases that
$\nilp$ is regular semisimple (eigenvalues of $\nilp$ are distinct and nonzero), and that
$\nilp$ is regular nilpotent (this includes $\Pet$).
In the former case, ${\rm Hess}( \nilp ,{\mathfrak h})$ is smooth. However, ${\rm Hess} (\nilp,{\mathfrak h})=\emptyset$ unless
${\mathfrak h}(i)\geq i$ and ${\mathfrak h}(i+1)\geq {\mathfrak h}(i)$. When
${\mathfrak h}$ satisfies these conditions, ${\mathfrak h}$ is called a \textbf{Hessenberg function}.
${\rm Hess}( \nilp ,{\mathfrak h})$ is
a {\bf Hessenberg variety} if ${\mathfrak h}$ is a Hessenberg function.

Generators for $I_{g,{\rm Hess}(\nilp,{\mathfrak h})}$ (up to radical) can be obtained by
a  modification of the construction of the generators for $I_{g,\Pet}$
described in Section~1.2 and Example~\ref{exa:petex}. (Small examples exist where these ``obvious'' generators do not define
a radical ideal.) The proof of Theorem~\ref{thm:main1} likely
generalizes to
show that if $\nilp$ is regular nilpotent, and ${\mathfrak h}$ is a Hessenberg function, each patch ideal of
${\rm Hess}(\nilp,{\mathfrak h})$ is a local complete intersection.

We checked that the Cohen-Macaulay claim of Conjecture~\ref{conj:Petmain} holds for all regular nilpotent Hessenberg varieties for $n\leq 4$,
most cases of $n\leq 5$ and a few cases of $n\leq 6$. Also, we made some checks of the upper-semicontinuity claim for $n\leq 4$.

Another example of Hessenberg varieties is
the {\bf Springer fiber} associated to a partition $\lambda$ of $n$.
Define ${\rm Springer}_{\lambda} = \left \{ F_{\bullet} \in GL_n/B : \nilp \cdot F_{i} \subset F_{i} \text{ for all } 1 \leq i \leq n \right \}$, where
$\nilp$ is
nilpotent matrix of Jordan type $\lambda$.
A key fact is that the irreducible components ${\rm Springer}_T$
of ${\rm Springer}_{\lambda}$ are indexed by standard Young tableaux $T$ of shape $\lambda$.
See, e.g., \cite{Fresse, PerSmi10} and the references for results on
singularities of these components.

We do not know if Conjecture~\ref{conj:Petmain} holds for ${\rm Springer}_{T}$'s. It is false for ${\rm Springer}_{\lambda}$ in general. For ${\rm Springer}_{\lambda}$, we checked
the Cohen-Macaulay claim of Conjecture~\ref{conj:Petmain}
for $n \leq 5$, and certain ``larger'' cases such as $\lambda = (3,2,1)$ and $\lambda=(2,2,1,1)$
(the latter case was the first known to have singular irreducible components \cite{Var79}).

\begin{example}
The aformentioned singular component of ${\rm Springer}_{(2,2,1,1)}$ found by \cite{Var79} is the unique
singular one
in that Springer fiber \cite{Fre}. It has
eight singular points from ${\rm Springer}_{(2,2,1,1)}\cap (GL_6/B)^T$, corresponding to
the permutations:
$135624, \ 135642, \ 136524$, $136542$, $315624, \ 315642, \ 316524,  \ 316542$.
On this component, these all have Hilbert-Samuel multiplicity $6$,
the same $h$-polynomial $1+4\uu+\uu^2$, and Gorenstein tangent cones.\qed
\end{example}

\subsection{Richardson varieties}
The {\bf Richardson variety} is defined by
$X_{u}^v={\overline{BuB/B}}\cap {\overline{B_{-}vB/B}}$.
This is empty unless $v\leq u$ and is of dimension $\ell(u)-\ell(v)$ otherwise, where $\ell(u)$ is
the Coxeter length of $u$.
It is the {\bf Schubert variety} $X_{u}={\overline{BuB/B}}$
when $v=id$ and the
{\bf opposite Schubert variety}
$X^v={\overline{B_{-}vB/B}}$
when $u=w_0$. Richardson varieties
are of importance in connection to Schubert calculus since
$[X_u^v] = [X_u] \cup [X_{w_0v}] = \sum_{x \in S_n} C_{u,w_0v}^{x} [X_{x}] \in H^*(GL_n/B; \mathbb Z)$,
where $C_{u,w_0v}^x$ is a {\bf generalized Littlewood-Richardson coefficient}.

From the definitions,
$wB\in X_{u}^v \mbox{ \ if and only if $v\leq w\leq u$ (in Bruhat order).}$
In this case, it is known $X_{u}^v$ is reduced, irreducible and Cohen-Macaulay; see, e.g., \cite{Brion:notes}.

The decomposition (P.III) is, in this setting:
\begin{eqnarray}\nonumber
X_{u}^v & = & \coprod_w (BwB/B\cap X_{u}^v)
 =  \coprod_w BwB/B\cap ({\overline{BuB/B}}\cap {\overline{B_{-}vB/B}})\\ \nonumber
& = & \coprod_w (BwB/B\cap {\overline{BuB/B}})\cap X^v
=  \coprod_{w\leq u} (BwB/B)\cap X^v.\nonumber
\end{eqnarray}
Each
${\mathcal J}_{w,X_u^v}=(BwB/B)\cap X^v$ is a {\bf Kazhdan-Lusztig variety}. This variety
is isomorphic, up to crossing by an affine space,
to the patch of the opposite Schubert variety $X^v$ \cite[Lemma A.4]{KL}. It
is singular in general.
The maximal torus $T$ in $B \cap B_{-}$ acts on each ${\mathcal J}_{w, X_u^v}$,
contracting each point towards $wB$. Hence $wB$ is the most singular point on that strata, in the sense of Lemma~\ref{lemma:goodcases}(II).

We want generators for $I_{w,X_{u}^v}$.
Let $J_{w,u}$ be the ideal in ${\mathbb C}[Z^{(w)}]$ generated by all size $1+r_{ij}^{u}$ minors of the southwest
$(n-i+1) \times j$ submatrix of $Z^{(w)}$. Here $r_{ij}^u$ is the rank matrix that counts the
number of $1$'s southwest of matrix coordinate $(i,j)$ in the permutation matrix for $u$. Let $J^{w,v}$
be defined by $1+{\overline r}_{ij}^{v}$
minors of the northwest $i\times j$ submatrix where ${\overline r}_{ij}^{v}$ is the number of $1$'s northwest of matrix
coordinate $(i,j)$ in the permutation matrix for~$v$.

The following is probably folklore, and is implicit in \cite[Lemma~A.4]{KL}:
\begin{proposition}
\label{prop:Richardson}
$I_{w,X_u^v}=J_{w,u}+J^{w,v}$.
\end{proposition}
\begin{proof}
Fulton~\cite{Fulton} showed that if $\pi:GL_n\to GL_n/B$ is the natural map, then the ideal defining $\pi^{-1}(X_u)$ is scheme-theoretically
generated by the size $1+r_{ij}^u$ minors of the southwest $(n-i+1)\times j$ submatrix.
Thus it follows from (\ref{eqn:patch}) that $I_{w,X_u}=J_{w,u}$. Similarly, $I_{w,X^v}=J^{w,v}$. Since $I_{w,X_u}$ scheme-theoretically cuts out an open neighbourhood of $wB$ in $X_u$ and $I_{w,X^v}$ does the same inside
$X^v$, then $I_{w,X_u}+I_{w,X^v}$ does it for
$wB\in X_u\cap X^v$. Now
$I_{w,X_u}+I_{w,X^v}$ is radical since $X_u\cap X^v$ is reduced.
The result then follows.
\end{proof}

\begin{problem}
Find an explicit Gr\"{o}bner basis for $I_{w,X_{u}^v}$.
\end{problem}
Computation suggests that there is a Gr\"{o}bner basis with
a squarefree monomial ideal.

\begin{conjecture}
\label{conj:reduced}
The tangent cones of $X_{u}^v$ are reduced.
\end{conjecture}

Using Proposition~\ref{prop:Richardson} and (P.V),
we verified Conjecture~\ref{conj:reduced} and positivity of the $h$-polynomial
exhaustively at $T$-fixed points,
for $n\leq 5$, and for many random singular cases for $6\leq n\leq 9$.
We also checked the tangent cones at $T$-fixed points are Cohen-Macaulay,
for $n\leq 4$. For \emph{Grassmannian} Richardson varieties, these claims follow from
\cite[Remark~7.6.6]{Kreiman}. These facts, combined with the results/checks made for
the Schubert case \cite{Li.Yong2}, leave us reasonably convinced. However, we have not made
substantive checks of upper-semicontinuity of the coefficients of $h_{g,X_{u}^v}(\uu)$
in this setting.


\subsection{Peterson-Schubert varieties}

Let
$R_w={\rm Pet}_n\cap X_w$.
This scheme (it is reducible in general) is $S$-invariant since $S$ acts on both ${\rm Pet}_n$ and $X_w$.

\begin{lemma}
\label{lemma:Rirred}
$R_{w_P}$ is irreducible and of dimension $n-k$, where $k$ is the number of blocks of $w_P$ (i.e. the $k$ of Equation (1.1)).
\end{lemma}
\begin{proof}
By Lemma~\ref{prop:Petjust}(VI), the
permutation matrix of $w_P$ has block matrix form
\begin{equation} \label{eqn:Pblock} w_P=C^{(1)}\oplus C^{(2)}\oplus\cdots\oplus C^{(k)}\end{equation}
where $C^{(j)}$ is an
$i_j\times i_j$ block matrix that is the ``reverse identity matrix'' of that size.
We claim (and this is a well-known fact) that
\begin{equation}
\label{eqn:aniso}
GL_{i_1} /B_{i_1} \times GL_{i_2}/B_{i_2} \times \cdots\times GL_{i_k}/B_{i_k} \cong X_{w_P}\cong P/B,
\end{equation}
where $i_1+\ldots+i_k=n$, cf.~(\ref{eqn:weyl}) and $B_{i_j}$ is the Borel subgroup of
upper triangular matrices inside $GL(i_j)$. To see this, let
$F_{\bullet}^{(t)}\in GL_{i_t}/B, \ \ 1\leq t\leq k$
and consider the map $\phi:GL_{i_1} /B_{i_1} \times GL_{i_2}/B_{i_2} \times \cdots\times GL_{i_k}/B_{i_k} \to X_{w_P}$ defined by
\[\phi: (F_{\bullet}^{(t)})_{t=1}^k \mapsto \left(F_{\bullet}^{(1)}, {\mathbb C}^{i_1}\oplus F_{\bullet}^{(2)},
{\mathbb C}^{i_1}\oplus{\mathbb C}^{i_2}\oplus F_{\bullet}^{(3)},\ldots, {\mathbb C}^{i_1}\oplus {\mathbb C}^{i_2}\oplus\cdots \oplus
{\mathbb C}^{i_{k-1}}\oplus F_{\bullet}^{(k)}\right)\in GL_n/B,\]
where
$V\oplus F_{\bullet}:=V\oplus F_1\subset V\oplus F_2\subset\cdots$.
We also made the natural identification of $F^{(1)}_{j}\subseteq {\mathbb C}^{i_1}$ with a subspace of
${\mathbb C}^n\cong \bigoplus_{j=1}^k {\mathbb C}^{i_j}$, etc.

Clearly, the image of $\phi$ is $X_{w_P}$ and that $\phi$ is an isomorphism.
Thus (\ref{eqn:aniso}) holds.

The map $\phi$ restricts to
$\psi:{\rm Pet}_{i_1}\times \cdots \times {\rm Pet}_{i_k}\to GL_n/B$.
The image sits inside $X_{w_P}$ since that is true of $\phi$.
For each $t$ we have $N\cdot ({\mathbb C}^{i_1}\oplus {\mathbb C}^{i_2}\oplus\cdots \oplus
{\mathbb C}^{i_t})\subset {\mathbb C}^{i_1}\oplus {\mathbb C}^{i_2}\oplus\cdots \oplus
{\mathbb C}^{i_t}$ (where ${\mathbb C}^{i_1}\oplus {\mathbb C}^{i_2}\oplus\cdots \oplus
{\mathbb C}^{i_t}$ is interpreted as a subspace of ${\mathbb C}^n$ in the manner indicated above).
Let $N^{(j)}$ be the standard regular nilpotent matrix for $GL(i_j)$. Then by assumption,
$N^{(j)}\cdot F_{r}^{(j)}\subset F_{r+1}^{(j)}$. Although of course the block matrix
$N^{(1)}\oplus N^{(2)}\oplus \cdots \oplus N^{(k)}$ is not equal to $N$, it is easy to
see from the definitions that the image of $\psi$ is inside $\Pet$.
Hence,
\begin{equation}
\label{eqn:factors}
\psi({\rm Pet}_{i_1}\times {\rm Pet}_{i_2}\times \cdots \times {\rm Pet}_{i_k}) \subseteq R_{w_P}
\end{equation}
(as sets). It remains to demonstrate the reverse containment. Now,
\[R_{w_P}=X_{w_P}\cap \Pet = \coprod_{w\leq w_P} (BwB/B) \cap \Pet=\coprod_{w_Q\leq w_P}
(B{w_Q}B/B)\cap \Pet=\coprod_{w_Q\leq w_P} U_Q w_QB/B,\]
where we have used Lemma~\ref{prop:Petjust} (VI) and (IV). We have also used the fact that
$w_Q B\in R_{w_P}$ if and only if $w_Q\leq w_P$.
In this case, the permutation matrix for $w_Q$ takes the form
\begin{equation} \label{eqn:Qblock} w_Q={\widetilde C}^{(1)}\oplus\cdots\oplus {\widetilde C}^{(k)} \end{equation}
where ${\widetilde C}^{(j)}$ occupies the same position
as the block $C^{(j)}$. Thus, the desired containment holds if we can show
$U_Q w_QB\subseteq \psi({\rm Pet}_{i_1}\times {\rm Pet}_{i_2}\times\cdots\times {\rm Pet}_{i_k})$
for each $w_Q\leq w_P$. However, this is clear from (\ref{eqn:Qblock}) and the definition of $\psi$.

Hence $R_{w_P}$ is irreducible, as each factor on the lefthand side of (\ref{eqn:factors}) is irreducible (and we are
working over ${\mathbb C}$).
The dimension claim for $R_{w_P}$ holds since ${\rm dim}({\rm Pet}_i)=i-1$.
\end{proof}

In view of Lemma~\ref{lemma:Rirred} and Theorem~\ref{thm:lciextended}(I) below, we call $R_{w_P}$ a
{\bf Peterson-Schubert variety}; cf. \cite{harada.tymoczko}. We have made checks of the Cohen-Macaulayness property of
Conjecture~\ref{conj:Petmain} for the varieties $R_{w_P}$ when $n\leq 5$.

Referring to (\ref{eqn:Qblock}),
let $w_Q^{(j)}$ denote the permutation in $S_{i_j}$ that corresponds to the matrix $\widetilde{C}^{(j)}$. The following extends Theorem~\ref{thm:singlocus} and Corollary~\ref{thm:first}.
\begin{theorem}
\label{thm:lciextended}
\begin{itemize}
\item[(I)] $R_{w_P}$ is a reduced, local complete intersection (and hence Cohen-Macaulay and Gorenstein).
\item[(II)] $R_{w_P}$ is singular if and only if $w_{P}$ contains the pattern $321$.
\item[(III)] The singular locus of $R_{w_P}$ is given by
${\rm Sing}(R_{w_P})=\coprod_{w_Q} {\mathcal J}_{w_Q,R_{w_P}}$,
where $w_Q\leq w_P$ and at least one of the $w_{Q}^{(j)}$ contains either of the patterns $123$ or $2143$.
\end{itemize}
\end{theorem}

\begin{proof}
(I): We want to show that the neighbourhood of $gB\in R_{w_P}$ given by ${\mathcal N}_{g,X_{w_P}}\cap {\mathcal N}_{g,\Pet}$ is
reduced and a complete intersection.
Since $gB\in \Pet$, then $gB=bw_QB$
for some $b\in U_Q$, by Lemma~\ref{prop:Petjust}(IV).
By the discussion of Section~2, we have that ${\mathcal N}_{g,\Pet}={\mathcal N}_{w_Q,\Pet}$ and similarly ${\mathcal N}_{g,X_{w_P}}={\mathcal N}_{w_Q,X_{w_P}}$. Therefore,
${\mathcal N}_{g,X_{w_P}}\cap {\mathcal N}_{g,\Pet}$
and ${\mathcal N}_{w_Q,X_{w_P}}\cap {\mathcal N}_{w_Q,\Pet}$ are equal.
So, it suffices to assume $g=w_Q$.

By the definition of $R_{w_P}=X_{w_P}\cap \Pet$ we have
$I_{w_Q,R_{w_P}}=I_{w_Q,X_{w_P}}+I_{w_Q,\Pet}$.
Referring to (\ref{eqn:Qblock}), we have the finer block decomposition
\begin{equation}
\label{eqn:finerdecomp}
{\widetilde C}^{(j)}=D^{(j,1)}\oplus\cdots\oplus D^{(j,f)}
\end{equation}
where each $D^{j,g}$ is a reverse identity matrix. That is,
${\widetilde C}^{(j)}= C^{(j)}$ if and only if $f=1$.

\begin{claim}
\label{claim:equalityofideals}
The following equality of ideals holds
\begin{equation}
\label{eqn:blocktolci}
I_{w_Q,X_{w_P}}+I_{w_Q,\Pet}=I_{w_Q,X_{w_P}}+\sum_{j=1}^{k} {\hat I}_{w_{Q}^{(j)}, {\rm Pet}_{i_j}},
\end{equation}
where ${\hat I}_{w_{Q}^{(j)}, {\rm Pet}_{i_j}}$ is the patch ideal for $w_{Q}^{(j)}\in {\rm Pet}_{i_j}$, and where we have embedded the ring
${\mathbb C}[Z^{(w_Q)}]$ into ${\mathbb C}[Z^{w_P}]$ using the variables of
$Z^{(w_P)}$ in the region occupied by ${\widetilde C}^{(j)}$ in the obvious manner.
\end{claim}
\begin{proof}
Recall that by Proposition~\ref{prop:Richardson}, the generators of $I_{w_Q,X_{w_P}}$ consist of all variables in the region
below the blocks $C^{(j)}$. Next, consider any generator $g_{k',j'}$ of $I_{w_Q,\Pet}$ as described by (\ref{eqn:theideal}). Suppose $j'$ is one of the columns
$a,a+1,\ldots,b$ occupied by $C^{(j)}$. Note that by the construction of $g_{k',j'}$,
row $k'$ cannot be strictly above the topmost row occupied by $C^{(j)}$ (since there is a $1$ to the left of column $j'$ and
in row $k'$). There are two cases:

\noindent
\emph{Case 1: $j'\leq b-2$}: (i) First suppose row $k'$ is strictly south of the bottom row occupied by $C^{(j)}$. Then since $j'\leq b-2$, we see $g_{k',j'}=\alpha_{j',1}z_{k',1}+\cdots+\alpha_{j',j'+1}z_{k',j'+1}-z_{k'+1,j'}\in I_{w_Q,X_{w_P}}$ since
$z_{k',1},\ldots,z_{k',j'+1},z_{k'+1,j'}\in  I_{w_Q,X_{w_P}}$.

(ii) If row $k'$ is one of the rows occupied by $C^{(j)}$, then
each variable of $Z^{(w_Q)}$, in that row and in columns $1,2,\ldots,a-1$,
 is in $I_{w_Q,X_{w_P}}$. Therefore, one can decompose $g_{k',j'}={\overline g}_{k',j'}+g_{k',j'}^{\vee}$ where ${\overline g}_{k',j'}\in I_{w_Q,X_{w_P}}$
and $g_{k',j'}^{\vee}$ is one of the generators (\ref{eqn:theideal}) of ${\hat I}_{w_{Q}^{(j)}, {\rm Pet}_{i_j}}$. Conversely, every generator of ${\hat I}_{w_{Q}^{(j)}, {\rm Pet}_{i_j}}$
gives rise to a generator of $I_{w_Q,\Pet}$ in this manner.

\smallskip
\noindent
\emph{Case 2: $j'=b-1$ or $j'=b$}: By the construction of $g_{k',j'}$, we see that row $k'$ must be strictly south of the rows occupied by $C^{(j)}$. Then $g_{k',j'}\in  I_{w_Q,X_{w_P}}$,
just as in Case 1(i).

The equality (\ref{eqn:blocktolci}) then follows from these considerations.
\end{proof}

$X_{w_P}$ is smooth and hence a local complete intersection. By Corollary~\ref{thm:first} each ${\rm Pet}_{i_j}$ is a complete intersection. Since Claim~\ref{claim:equalityofideals} expresses $I_{w_Q,X_{w_P}}+I_{w_Q,\Pet}$ as
a sum of ideals on disjoint variable sets, each ideal being a radical, complete intersection, (I) holds.

(II): First suppose that $w_P$ contains the pattern $321$.  Hence, one of the factors ${\rm Pet}_{i_j}$ in
\[X_{w_P} \cap {\rm Pet}_n \cong {\rm Pet}_{i_1} \times {\rm Pet}_{i_2} \times \cdots \times {\rm Pet}_{i_k}\]
is isomorphic to the singular ${\rm Pet}_m$ for some $m \geq 3$.
Conversely, suppose that $R_{w_P}$ is singular.  Therefore one of the factors ${\rm Pet}_{i_j}$ is singular (since a product of smooth varieties is smooth).  Since $\rm{Pet_2}$ is smooth, we deduce that $i_i \geq 3$.
Hence $w_P$ contains $321$.

(III): The discussion of (I), and specifically (\ref{eqn:blocktolci}), imply
\[{\mathcal N}_{g, R_{w_P}} \cong
{\mathcal N}_{g^{(1)},{{\rm Pet}_{i_1}}}\times
{\mathcal N}_{g^{(2)},{{\rm Pet}_{i_2}}}\times\cdots\times
{\mathcal N}_{g^{(k)},{{\rm Pet}_{i_k}}},\]
where $g=bw_Q$ for $b\in U_Q=B^{(1)}\oplus B^{(2)}\oplus\cdots\oplus B^{(k)}$ where each $B^{(j)}\in U_{Q^{(j)}}$
and $Q^{(j)}$ is the parabolic subgroup of $GL_{i_j}$ corresponding to $w_Q^{(j)}$.
Writing $b=b_1\oplus b_2\oplus \cdots \oplus b_k$ where $b_j\in B^{(j)}$, in the above isomorphism we set
$g^{(j)}=b_j\cdot w_Q^{(j)}$. Now
${\mathcal N}_{{g}, R_{w_P}}$ is smooth if and only if each ${\mathcal N}_{g^{(j)},{{\rm Pet}_{i_j}}}$ is smooth.
Then the claim follows from Theorem~\ref{thm:singlocus}.
\end{proof}

\section*{Acknowledgements}
We thank Li Li for questions that
motivated us to initiate this project, as well as for
other highly useful remarks. We are grateful to Allen Knutson and Alexander Woo for in depth discussion/comments on a draft. We thank Anders Buch,
Jim Carrell, Elizabeth Csima, Alex Fink, William Haboush, Megumi Harada, June Huh, Mannoj Kummini, Kostanze Rietsch, Hal Schenck,
Alexandra Seceleanu and Julianna Tymoczko
for very helpful comments and discussions. We also thank the anonymous referees for
detailed comments to an earlier version, including simplifications/corrections to Section~7.3.
This work was partially completed during two workshops at the American
Institute of Mathematics, the first on ``Localization techniques in equivariant cohomology'' in March 2010
and the second on ``Real algebraic systems'' in October 2010.  We made use of {\tt SAGE}, {\tt Singular} and {\tt Macaulay 2} in our
computations.
EI was partially supported by NSF VIGRE grant DMS-0602242. AY was partially supported by NSF grant DMS-0901331.

\end{document}